\documentclass[11pt,oneside,english]{amsart}

\usepackage[margin=1.3in]{geometry}
\usepackage[dvipsnames]{xcolor}
\usepackage{blkarray}
\usepackage{paralist}

\usepackage[backref=page,colorlinks,citecolor=blue,bookmarks=true]{hyperref}\hypersetup{linkcolor=blue,  citecolor=blue, urlcolor=blue}

\definecolor{DarkGreen}{rgb}{0.1,0.5,0.1}
\definecolor{DarkRed}{rgb}{0.5,0.1,0.1}
\definecolor{DarkBlue}{rgb}{0.1,0.1,0.5}

\usepackage{mleftright}

\usepackage[small]{caption}
\usepackage[backref=page,colorlinks,citecolor=blue,bookmarks=true]{hyperref}\hypersetup{linkcolor=blue,  citecolor=blue, urlcolor=blue}
\usepackage{amssymb,amsmath,amsthm,amsfonts,enumitem}
\usepackage{nicefrac,comment}
\usepackage{tikz}
\usepackage{tikz-cd}
\usetikzlibrary{arrows,decorations.pathmorphing,decorations.shapes,snakes,patterns,positioning}
\usepackage{thmtools}
\usepackage{thm-restate}
\usepackage{graphicx}
\usepackage{bbm}
\usepackage[capitalise]{cleveref}
\usepackage{expdlist}
\usepackage{setspace}
\usepackage{algorithm}
\usepackage[noend]{algpseudocode}
\usepackage{caption}

\newcommand{\deffont}[1]{\textnormal{\textsf{#1}}} 

\newcommand{\C}{\ensuremath{\mathbb{C}}}  
\newcommand{\cM}{\ensuremath{\mathcal{M}}}

\newcommand{\R}{{\mathbb R}}
\newcommand{\Z}{{\mathbb Z}}
\newcommand{\ZZ}{{\mathbb Z}}

\newcommand{\RR}{{\mathbb R}}

\newcommand{\N}{\ensuremath{\mathbb{N}}}
\newcommand{\PP}{\ensuremath{\mathbb{P}}}
\mathchardef\mhyphen="2D

\newcommand{\inabs}[1]{\left|#1\right|}

\newcommand{\inset}[1]{\left\{#1\right\}}
\newcommand{\inabset}[1]{\inabs{\inset{#1}}}

\newcommand{\inparen}[1]{\left(#1\right)}

\newcommand{\inang}[1]{\left\langle#1\right\rangle}

\newcommand{\ol}{\overline}

\newcommand{\eps}{\varepsilon}
\renewcommand{\epsilon}{\varepsilon}

\newif\ifcomments
\commentstrue

\theoremstyle{plain}
\declaretheorem[name=Theorem,numberwithin=section]{theorem}
\declaretheorem[name=Lemma,sibling=theorem]{lemma}
\newtheorem*{lemma*}{Lemma} 
\newtheorem*{theorem*}{Theorem} 
\newtheorem{definition}[theorem]{Definition}
\newtheorem{defn}[theorem]{Definition}

\newtheorem{corollary}[theorem]{Corollary} 

\newtheorem{remark}[theorem]{Remark}

\newtheorem{proposition}[theorem]{Proposition}
\newtheorem{prop}[theorem]{Proposition}

\declaretheorem[name=Problem]{problem}
\newtheorem{conjecture}[problem]{Conjecture}

\Crefname{ourtheorem}{Theorem}{Theorems}

\Crefname{ourcorollary}{Corollary}{Corollary}
\Crefname{property}{Property}{Properties}
\Crefname{condition}{Condition}{Conditions}
\Crefname{observation}{Observation}{Observations}
\Crefname{step}{Step}{Steps}

\Crefname{customthm}{Theorem}{Theorems}

\Crefname{customlem}{Lemma}{Lemmas}

\newtheoremstyle{named}{}{}{\itshape}{}{\bfseries}{.}{.5em}{\thmnote{#3}}
\theoremstyle{named}

\usepackage{mathtools}

\DeclareMathOperator{\Out}{Out}%
\DeclareMathOperator{\id}{id}%
\DeclareMathOperator{\Sym}{Sym}%
\DeclareMathOperator{\stab}{Stab}%
\DeclareMathOperator{\GL}{GL}%

\DeclareMathOperator{\Aut}{Aut}%
\DeclareMathOperator{\PSL}{PSL}%
\DeclareMathOperator{\Hom}{Hom}%
\DeclareMathOperator{\Sp}{Sp}%
\DeclareMathOperator{\T}{(T)}%

\DeclareMathOperator{\MCG}{MCG}%
\DeclareMathOperator{\Sol}{Sol}%
\DeclareMathOperator{\BS}{BS}%
\DeclareMathOperator{\IRS}{IRS}%
\DeclareMathOperator{\Sub}{Sub}%
\DeclareMathOperator{\findex}{f.i.}%
\DeclareMathOperator{\ptau}{(\tau)}

\DeclareMathOperator{\Emp}{Empirical}

\DeclareMathOperator{\eperiod}{\,\,\text{.}}
\DeclareMathOperator{\ecomma}{\,\,\text{,}}

\DeclareMathOperator{\TV}{TV}%
\newcommand{\dTV}{d_{\TV}}%

\DeclareMathOperator{\UB}{UB}%

\mathchardef\shortdash="2D

\DeclareMathOperator{\flex}{\shortdash flex}

\newcommand{\SAS}{\textup{\texttt{Sample and Substitute}}}
\newcommand{\LSM}{\textup{\texttt{Local Statistics Matcher}}}

\newcommand{\varX}{\mathsf{X}}
\newcommand{\varY}{\mathsf{Y}}

\newcommand{\calU}{\mathcal{U}}
\newcommand{\calH}{\mathcal{H}}

\usepackage{amsrefs}
\usepackage[foot]{amsaddr}

\commentsfalse

\begin{document}
	
\title{Testability in group theory}

\author{Oren Becker}
\address{Oren Becker, Centre for Mathematical Sciences,
Wilberforce Road, Cambridge CB3 0WA, United Kingdom}
\email{oren.becker@gmail.com}

\author{Alexander Lubotzky}
\address{Alexander Lubozky, Department of Mathematics, Weizmann Institute of Science, Rehovot, Israel}
\email{alex.lubotzky@mail.huji.ac.il}

\author{Jonathan Mosheiff}
\address{Jonathan Mosheiff, Computer Science Department, Carnegie Mellon University,
	Pittsburgh, PA, USA}
\email{jonathanmush@gmail.com}

\date{}
\maketitle

\begin{abstract}
This paper is a journal counterpart to \cite{BLM}, in which we initiate the study of property testing problems concerning a finite system of relations $E$ between permutations, generalizing the study of stability in permutations.
To every such system $E$, a group $\Gamma=\Gamma_E$ is associated and the testability of $E$ depends only on $\Gamma$ (just like in Galois theory, where the solvability of a polynomial is determined by the solvability of the associated group).
This leads to the notion of \emph{testable} groups, and, more generally, \emph{Benjamini--Schramm rigid} groups.
The paper presents an ensemble of tools to check if a given group $\Gamma$ is testable/BS-rigid or not.
\end{abstract}

\section{Introduction}\label{sec:intro}
This paper is a journal counterpart to \cite{BLM},
which appeared in
the proceedings of the 2021 IEEE Annual Symposium on Foundations of
Computer Science (FOCS). In that paper we initiated a systematic
study of testability of relations between permutations. That paper was written from the point of view of property testing---a core subject in theoretical computer science. Here, we present the content of \cite{BLM} (plus some supplements) from a very different perspective.
While that paper was written in a mainly combinatorial language, the current
paper is mostly group theoretic. We will explain how the results of
\cite{BLM} can be viewed (and proved) as group theoretic statements, and through the language of invariant random subgroups. We
hope that this presentation will attract the group theory community
to join this line of research.

From an algorithmic point of view, we are interested in problems where one is given several permutations, and wishes to determine whether they satisfy a given system of equations or are far from doing so. If this can be done by an algorithm that only queries a constant number of entries of the given permutations, we say that the system of equations is \deffont{testable}. For example, consider the system consisting of the single equation $\mathsf{XY=YX}$. The corresponding algorithmic problem is, given two permutations $A$ and $B$ in the symmetric group $\Sym(n)$, to distinguish between the following two cases: (i) $A$ and $B$ commute, or (ii) the pair $(A,B)$ is far, under the normalized Hamming metric, from every commuting pair $(A',B')\in \Sym(n)\times\Sym(n)$.

We begin by reformulating the main definitions of \cite{BLM} in a more group-oriented language.
Throughout the introduction, we
fix a finite set $S = \{s_1,\dotsc, s_k\}$, and let $F_S$ be the free group on $S$.
For a word $w\in F_S$, write $|w|$ for the length of the reduced form of $w$.
For $n\in \N$,
write $\Sym(n)$ for the symmetric group on $[n]\coloneqq\inset{1,\dotsc,n}$. For a tuple of permutations $\ol\sigma = \inparen{ \sigma_1,\dotsc, \sigma_k}\in \Sym(n)^k$, let $w(\ol\sigma)$ denote the image of $w$ under the homomorphism from $F_S$ to $\Sym(n)$ that takes $s_i$ to $\sigma_i$ for all $1\le i\le k$. It will often be useful to think of $\ol\sigma$ as describing a directed graph $G_{\ol \sigma}$ with edges colored $\{1,\dotsc, k\}$, where the vertex set is $[n]$ and for every $1\le i\le k$ and $x\in [n]$ there is an $i$-colored edge from $x$ to $\sigma_i x$. 

Let $E$ be a subset of $F_S$.
Then $E$ gives rise to a system of equations $\inset{w = 1\mid w\in E}$, and one may consider the space of solutions in any group $G$. We are interested in the space of solutions for $E$ in $\Sym(n)$, namely,
\[
\Sol_{E}(n):=\inset{\ol\sigma\in \Sym(n)^k \mid\forall w\in E~~w(\ol\sigma)=\id }\eperiod 
\]
We do not specify the set $S$ in the notation $\Sol_E(n)$ since it will always be clear from the context.

The group $\Sym(n)$ is endowed with the normalized Hamming metric\footnote{We shall omit the subscript $n$ when it is clear from the context.}
\begin{equation}\label{eq:PermutationMetric}
d_{n}(\sigma,\tau):=\frac{1}{n}\left|\left\{ x\in\left[n\right]\mid\sigma x\neq\tau x\right\} \right|\eperiod \end{equation}
For every $\eps>0$, we define
\[
\Sol_{E}^{<\eps}(n):=\inset{ \ol \sigma\in \Sym(n)^k\mid \exists \ol\tau\in \Sol_{E}(n),~\sum_{i=1}^k d\inparen{\sigma_i, \tau_i}<\eps}\ecomma\]
and $\Sol_{E}^{\geq\eps}(n):=\Sym(n)^k\setminus\Sol_E^{<\eps}(n)$. 

We also denote
$\Sol_E = \bigcup_{n\in \N}\Sol_E(n) $, $\Sol_E^{\ge \eps} = \bigcup_{n\in \N}\Sol_E^{\ge \eps}(n)$ and $\Sol_E^{< \eps} = \bigcup_{n\in \N}\Sol_E^{< \eps}(n)$.

\begin{definition}[Algorithmic separation]\label{def:separation}
Fix two disjoint sets $A,B\subseteq \bigcup_{n\in\N}\Sym(n)^k$. An \deffont{$(A,B)$-separator} is a randomized algorithm (that is, an algorithm allowed to make random decisions) $\cM$ which takes as input an integer $n\in \N$ and a tuple of permutations $\ol \sigma\in \Sym(n)^k$, and has the following properties.
\begin{itemize}
    \item \textbf{Completeness:} If $\ol\sigma\in A$ then $\cM$ accepts with probability at least $0.99$.
    \item \textbf{Soundness:} If $\ol \sigma\in B$ then $\cM$ rejects with probability at least $0.99$.
    \item \textbf{Query efficiency:} There exists some $q\in \N$ such that, on every input, the algorithm $\cM$ makes at most $q$ queries, each of the form ``what is $ \sigma_ix$'' or ``what is $ \sigma_i^{-1}x$'' for some $1\le i\le k$ and $1\le x\le n$. Crucially, $q$ is not allowed to depend on $n$.
\end{itemize}
In this case, we also say that $\cM$ has \deffont{query complexity} $q$ and that the pair $(A,B)$ is \deffont{separable} (with $q$ queries).
\end{definition}

\begin{definition}\label{def:testable}
Fix a subset $E$ of $F_S$. If for every $\eps > 0$ the pair $(\Sol_E, \Sol_E^{\ge \eps})$ is separable, we say that $E$ is \deffont{testable}.
\end{definition}

\cref{def:testable} originates in \cite{BLM}. The main goal in this new line of research is identifying which systems of equations are testable. See \cref{sec:examples} for concrete examples.

The following algorithm, introduced in \cite{BeckerMosheiff,BLM}, is a natural attempt at producing a separator for $(\Sol_E, \Sol_E^{\ge \eps})$ when $E$ is finite.
\begin{algorithm}[H]\caption{\label{alg:SAS}
		\textsf{Sample and Substitute} with word set $E$ and repetition factor $s\in \N$\\
		\textbf{Input:} $n\in\N$ and $\ol \sigma\in \Sym(n)^k$}\begin{algorithmic}[1]
		
		\State Sample $\left(w_1,x_1\right),\dotsc,\left(w_s,x_s\right)$
		uniformly and independently from $E\times\left[n\right]$.
		
		\If{ $w_j(\ol\sigma)x_j=x_{j}$ for all $1\leq j\leq s$}
		
		\State Accept.
		
		\Else\State Reject.\EndIf
		
	\end{algorithmic}
\end{algorithm}
Note that \SAS{} makes at most $C\cdot s$ queries, $C=C(E)$. Also note that \SAS{} has perfect completeness, namely, if $\ol\sigma$ belongs to $\Sol_E(n)$, the algorithm accepts deterministically, and not only with probability at least $0.99$. As we shall explain later, for $E=\inset{\varX\varY\varX^{-1}\varY^{-1}}$, \SAS{} is indeed a separator for $(\Sol_E, \Sol_E^{\ge \eps})$ (for an appropriate choice of the repetition factor $s=s(\eps)$), showing that $E$ is testable (this is a reformulation of \cite{ArzhantsevaPaunescu} in an algorithmic language).

The \SAS{} algorithm is related to the extensively studied notion of \emph{stability in permutations} via the following lemma.

\begin{prop}
\label{prop:stability-defs}
    Let $E$ be a finite subset of $F_S$, and write $\Gamma=\inang{S\mid E}$ for the group presented by the generators $S$ and the relators $E$. The following conditions are equivalent.
    \begin{enumerate}
        \item\label[condition]{enum:StabilityBySAS} For every $\eps > 0$ there exists $s\in \N$ such that \SAS{} with word set $E$ and repetition factor $s$ is a $(\Sol_E,\Sol_E^{\ge \eps})$-separator.
        \item\label[condition]{enum:StabilityCombinatorial} For every $\eps > 0$ there exists $\delta > 0$ such that for every $n\in \N$ and $\ol\sigma \in \Sym(n)^k$, if
        \begin{equation} \label{eq:AlmostSolution}
            \frac{1}{\left|E\right|}\sum_{w\in E}d(w(\ol\sigma), \id)< \delta
        \end{equation} 
        then $\ol\sigma \in \Sol_{E}^{<\eps} $.
        \item\label[condition]{enum:StabilityByHoms} For every sequence of functions $\inparen{f_\ell:\Gamma\to \Sym(n_\ell)}_{\ell=1}^\infty$, $n_\ell\to\infty$, such that
        \begin{equation}\label{eq:AsymptoticHomomorphism}
        d\inparen{f_\ell(\gamma_1\gamma_2),f_\ell(\gamma_1)f_\ell(\gamma_2)}\to_{\ell\to \infty} 0\qquad\forall \gamma_1,\gamma_2\in \Gamma\ecomma    
        \end{equation}
        there exists a sequence of homomorphisms $\inparen{h_\ell:\Gamma\to \Sym(n_\ell)}_{\ell=1}^\infty$ such that
        \begin{equation}\label{eq:AsymptoticallyCloseFuncs}
        d\inparen{f_\ell(\gamma),h_\ell(\gamma)}\to_{\ell\to\infty} 0\qquad\forall \gamma\in \Gamma\eperiod    
        \end{equation}
        \end{enumerate}
\end{prop}

\begin{proof}
The equivalence of \cref{enum:StabilityBySAS,enum:StabilityCombinatorial} is straightforward consequence (cf. \cite{BeckerMosheiff,BLM}) of the fact that the left-hand side of \eqref{eq:AlmostSolution} is equal to the probability that \SAS{} with word set $E$ and repetition factor $1$ rejects $\ol\sigma$. The equivalence of \cref{enum:StabilityCombinatorial,enum:StabilityByHoms} is also well known (see, e.g., \cite[Lem.\ 3.1]{Ioana}).
\end{proof}
\begin{definition}\label{def:stability}~
\begin{enumerate}
    \item\label{enum:FiniteSubsetStable} A finite subset $E$ of $F_S$ is \deffont{stable} (in permutations) if it satisfies the equivalent conditions stated in \cref{prop:stability-defs}.
    \item A group $\Gamma$ (not necessarily finitely presented) that satisfies \cref{enum:StabilityByHoms} of \cref{prop:stability-defs} is said to be
        \deffont{stable} (in permutations). 
\end{enumerate}
\end{definition}
A tuple of permutations $\ol \sigma$ that satisfies \eqref{eq:AlmostSolution} can be thought of as an \emph{approximate solution} for $E$. Hence, essentially, $E$ is stable if every approximate solution for $E$ is close a solution. In \cref{rem:GroupPropertyInfinite}, we will see that the notion of a stable subset $E$ of $F_S$ (\cref{def:stability}(\ref{enum:FiniteSubsetStable})) can also be defined without assuming that $E$ is finite.  
\begin{remark}\label{rem:finitelyPresentedStable}
Let $E$ be a finite subset of $F_S$. By \cref{prop:stability-defs},
$E$ is stable if and only if the same is true for the group $\inang{S\mid E}$, transforming the algorithmic question of the stability of $E$ into a group-theoretic question
about the stability of the finitely-presented group $\inang{S\mid E}$.
\end{remark}

By \cref{def:stability} and \cref{prop:stability-defs}(\ref{enum:StabilityBySAS}), every stable subset of $F_S$ is testable.
We will show below that the converse is not true.
In other words, there are subsets of $F_S$ which are testable, but require a different sort of separator than \SAS{}. For example, $E_{m,n}=\inset{\varX\varY^m\varX^{-1}\varY^{-n}}$ (with $S = \inset{\varX,\varY}$)), corresponding to the equation $\varX\varY^m=\varY^n\varX$, is such a system whenever $m\geq 2$ and $n\geq 2$ are co-prime integers (see \cref{sec:examples}).
At this point, the reader may wonder whether every system of equations is testable.
However, as we shall see, when $m=n\geq 2$, the set $E_{m,n}$ is non-testable. The proof of non-testability relies on a universality property, which we discuss below.

We present another algorithm, named \LSM, which is a \emph{universal} separator. Namely, in contrast to \SAS, every testable finite subset of $F_S$ can be tested by \LSM{} separators. To define \LSM{} we need some notation. Given two distributions $\theta_1$ and $\theta_2$ over a finite set $\Omega$, denote their \deffont{total-variation distance} by 
$$\dTV(\theta_1,\theta_2) = \frac 12 \sum_{x\in \Omega}\inabs{\theta_1(x)-\theta_2(x)}\eperiod$$
Given $\ol\sigma\in \Sym(n)^k$ and $x\in [n]$, let $$\stab_{\ol\sigma}(x) = \inset{w\in F_S\mid w(\ol\sigma)x = x}\eperiod$$
Given a finite set $P\subseteq F_S$, let $N_{\ol \sigma,P}$ denote the distribution, over the power set $2^P$, of $\stab_{\ol\sigma}(x) \cap P$, where $x$ is sampled uniformly from $[n]$. 

\begin{remark}\label{rem:StabEncodesBalls}
For intuition, consider the case where $P$ consists of all words in $F_S$ of length at most some fixed even $r\in \N$. Then, the set $\stab_{\ol \sigma}(x)\cap P$ yields a description (up to vertex relabeling) of the ball of radius $r/2$ centered at the vertex $x$ in the graph $G_{\ol \sigma}$. Hence, $N_{\ol \sigma, P}$ encodes the distribution of the isomorphism class (as a directed edge-labelled rooted graph) of a ball of radius $r/2$ in $G_{\ol \sigma}$, centered at a vertex sampled uniformly from the set of vertices of $G_{\ol\sigma}$.
\end{remark}

\begin{algorithm}[H]\caption{\label{alg:LSM}
		\textsf{Local Statistics Matcher} for $E$ with repetition factor $s\in \N$, finite word set $P\subset F_S$ and proximity parameter $\delta > 0$\\
		\textbf{Input:} $n\in\N$ and $\ol\sigma\in \Sym(n)^k$}\begin{algorithmic}[1]
		
 		\State\label[step]{step:LSM-sampling}Sample $x_{1},\dotsc,x_{s}$ uniformly
 		and independently from $\left[n\right]$.
		
 		\State For each $1\le j\le s$, compute the set $\stab_{\ol\sigma}\left(x_{j}\right)\cap P$
 		by querying $\ol\sigma$.
		
 		\State Let $N_{\ol\sigma,P}^{\Emp}$ be the distribution of $\stab_{\ol\sigma}\left(x_{j}\right)\cap P$
 		where $j$ is sampled uniformly from $\left[s\right]$.
		
 		\If{ 
 			\begin{equation}
 				\min\left\{ d_{\TV}\left(N_{\ol\sigma,P}^{\Emp},N_{\ol\tau,P}\right)\mid \ol\tau\in\Sol_{E}(n) \right\} \le\delta\label{eq:LSMCondition}
 			\end{equation}
 		} 
	
 		\State Accept.
		
 		\Else\State Reject.\EndIf
	\end{algorithmic}
\end{algorithm}
Note that \cref{step:LSM-sampling} can be implemented by computing $w(\ol\sigma)(x)$ for each $w\in P$. Thus \LSM{} makes at most $C\cdot s$ queries, $C=C(E)$. However, the time complexity of \LSM{} might be higher due to the computation involved in \cref{step:LSM-sampling}. This is unlike the case of \SAS{}, where the time complexity is proportional to the query complexity. It is an interesting problem to find implementations of \cref{step:LSM-sampling} with small time complexity, see \cref{sec:TimeComp}.

Note that the distribution $N_{\ol\sigma,P}^{\Emp}$, computed during a run of \LSM{}, is an approximation of the distribution $N_{\ol \sigma, P}$, where the former is obtained by taking $s$ independent samples from the latter. The quality of this approximation improves as the parameter $s$ grows. Using $N_{\ol\sigma,P}^{\Emp}$, checking whether Inequality \eqref{eq:LSMCondition} holds is an attempt to determine whether $N_{\ol \sigma, P}$ is close to the distribution $N_{\ol \tau, P}$ for some  solution $\ol\tau$ for $E$.

Before formulating our claim about the universality of \LSM{} (\cref{thm:LSMUniversal}), we introduce a geometric-combinatorial definition, which will turn out to be the main proxy through which we study testability in the rest of this paper.
\begin{definition}\label{def:BS-rigid}
	A subset $E$ of $F_S$ is \deffont{Benjamini--Schramm-rigid}\footnote{In \cite{BLM}, the notion of \emph{Benjamini--Schramm-rigidity} is termed \emph{statistical-distinguishability}.} (or \deffont{BS-rigid} for short) if for every $\eps > 0$ there exist a finite set $P = P(\eps)\subseteq F_S$ and $\delta = \delta(\eps) > 0$ such that $\dTV\inparen{N_{\ol \sigma,P}, N_{\ol \tau,P}} \ge \delta$ for every $n\in \N$, $\ol \sigma \in \Sol_{E}^{\ge \eps}(n)$ and $\ol \tau \in \Sol_{E}(n)$. In this case we also say that $E$ is $\inparen{P(\eps),\delta(\eps)}$-BS-rigid.
\end{definition}

\begin{theorem}[{\cite[Thm.\ 3]{BLM}}, Universality of \LSM]
\label{thm:LSMUniversal}
The following conditions are equivalent for a finite subset $E$ of $F_S$.
	\begin{enumerate}
		\item \label{enum:UniversalityTestable}$E$ is testable.
		\item \label{enum:UniversalityLSM} For every $\eps > 0$ there exist $s\in \N$, a finite set $P\subseteq F_S$ and $\delta > 0$ such that \LSM{} with repetition factor $s$, word set $P$ and proximity parameter $\delta$ is a $\inparen{\Sol_E, \Sol_E^{\ge \eps}}$-separator.
		\item \label{enum:UniversalityRigid} $E$ is BS-rigid.
	\end{enumerate}
\end{theorem}

\cref{thm:LSMUniversal}, reduces the algorithmic question of testability of a finite subset $E$ of $F_S$ to the equivalent geometric question of whether $E$ is BS-rigid.
Henceforth, we focus on BS-rigidity (without assuming that $E$ is finite).
We start by proving an analogue of \cref{prop:stability-defs} for BS-rigidity (rather than stability).

For a group $\Gamma$, a function $f\colon\Gamma\to\Sym(n)$ and $x\in[n]$, let $\stab_f(x)=\inset{\gamma\in\Gamma\mid f(\gamma)x=x}$ (note that $\stab_f(x)$ is usually not a group).
For a finite subset $Q$ of $\Gamma$, let $N_{f,Q}$ be the probability distribution, over the power set $2^Q$, of $\stab_f(x)\cap Q$ where $x$ is sampled uniformly from $[n]$.

\begin{prop}
\label{prop:BS-rigidViaHoms}
Let $E\subseteq F_S$ and write $\Gamma=\langle S \mid E \rangle$. The following conditions are equivalent.
\begin{enumerate}
    \item $E$ is BS-rigid.
    \item\label[condition]{enum:BS-rigidityByHoms}
    For every sequence of functions $\inparen{f_\ell\colon\Gamma\to \Sym(n_\ell)}_{\ell=1}^\infty$, $n_\ell\to\infty$, such that
    $$d\inparen{f_\ell(\gamma_1\gamma_2),f_\ell(\gamma_1)f_\ell(\gamma_2)}\to_{\ell\to \infty} 0 \qquad\forall \gamma_1,\gamma_2\in \Gamma\ecomma$$
    if there is a sequence of homomorphisms $\inparen{g_\ell\colon\Gamma\to \Sym(n_\ell)}_{\ell=1}^\infty$, such that
    $$d_{\TV}(N_{f_\ell,Q},N_{g_\ell,Q})\to_{\ell\to\infty} 0\qquad\forall Q\subset\Gamma, |Q|<\infty$$
    then there is a sequence of homomorphisms $\inparen{h_\ell\colon\Gamma\to \Sym(n_\ell)}_{n_\ell=1}^\infty$ such that
    $$d\inparen{f_{\ell}(\gamma),h_{\ell}(\gamma)}\to_{\ell\to\infty} 0\qquad\forall \gamma\in \Gamma\eperiod$$
\end{enumerate}
\end{prop}
\cref{prop:BS-rigidViaHoms} is a special case of \cref{prop:BS-rigidViaHomsFlex} below. The latter is proved in \cref{sec:ProofBS-rigidViaHoms}.
\begin{definition}
A group $\Gamma$ is \deffont{BS-rigid} if it satisfies \cref{enum:BS-rigidityByHoms} of \cref{prop:BS-rigidViaHoms}.
\end{definition}
\begin{remark}
In light of the equivalence between (\ref{enum:UniversalityTestable}) and (\ref{enum:UniversalityRigid}) in \cref{thm:LSMUniversal}, one may be tempted to call a group satisfying \cref{enum:BS-rigidityByHoms} of \cref{prop:BS-rigidViaHoms} \emph{testable}. However, the interpretation of BS-rigidity as testability (\cref{thm:LSMUniversal}), while valid for a finite set $E$, might not be valid for an infinite $E$. Thus, we focus on the geometric notion of BS-rigidity instead of getting into subtle issues related to computability that may arise when studying testability of infinite sets.
\end{remark}
In \cref{sec:IRSFormulation} we give an equivalent definition for a BS-rigid group and a stable group, from the point of view invariant random subgroups.
The study of stability from this point of view has been fruitful
\cites{BLT,LevitLubotzky1,LevitLubotzky2,Zheng}, and we believe the same will be true for BS-rigidity.

A basic observation,
which follows at once from the way we presented the definitions and results above, is that the BS-rigidity
(resp. stability) of $E$ depends only on the isomorphism class of the group $\langle S \mid E \rangle$.
Namely,
\begin{prop}\label{prop:GroupProperty}
Fix finite sets $S_1$ and $S_2$,
and take finite sets $E_{1}\subset F_{S_1}$ and $E_{2}\subset F_{S_2}$
such that the groups $\langle S_1 \mid E_1 \rangle$ and $\langle S_2 \mid E_2 \rangle$ are isomorphic.
Then
\begin{enumerate}
\item \label{enum:GroupPropertyDistniguishable} \cite[Prop.\ 3.2]{BLM} $E_{1}$ is BS-rigid if and only if $E_{2}$ is BS-rigid.
\item \label{enum:GroupPropertyStable}\cite{ArzhantsevaPaunescu} $E_{1}$ is stable if and only if $E_{2}$
is stable.
\end{enumerate}
\end{prop}
\begin{proof}
1) This claim is proved in \cite[Prop.\ 3.2]{BLM}. It also follows from \cref{prop:BS-rigidViaHoms}.

2) This was observed in \cite{ArzhantsevaPaunescu} (using Tietze transformations), and also follows from the equivalence of (1) and (3) in \cref{prop:stability-defs}.
\end{proof}

\begin{remark}\label{rem:GroupPropertyInfinite}
By \cref{prop:BS-rigidViaHoms}, \cref{prop:GroupProperty}(\ref{enum:GroupPropertyDistniguishable}) remains true even without the assumption that $E_1$ and $E_2$ are finite.

In fact, stability can also be defined for an infinite subset $E$ of $F_S$ (\cite[Lem.\ 3.1]{Ioana} or \cite[Def.\ 3.11]{BLT}). Under this definition, \cref{prop:GroupProperty}(\ref{enum:GroupPropertyStable}) remains true even without the assumption that $E_1$ and $E_2$ are finite. The algorithmic definition of stability of a (possibly infinite) set $E$ is: $E$ is \deffont{stable} if for every $\eps > 0$ there exist $s\in \N$ and a finite subset $E_0$ of $E$ such that \SAS{} with word set $E_0$ and repetition factor $s$ is a $\inparen{\Sol_E,\Sol_E^{\ge \eps}}$-separator.
\end{remark}

Motivated by the algorithmic interpretation of BS-rigidity as testability, we wish to develop an ensemble of tools which can check if a given
group is BS-rigid (resp. stable) or not. While stability in permutations has been studied extensively in recent years (see 
\cite{ArzhantsevaCherix,ArzhantsevaPaunescu,BeckerLubotzky,BLT,BeckerMosheiff,BowenBurton,GlebskyRivera,Ioana,LLM,LazarovichLevit,LevitLubotzky1,LevitLubotzky2,MoralesGlebsky,ThomICM,Zheng}, the present paper (together with its counterpart \cite{BLM}) is the first about testability and BS-rigidity.

The following theorem fully characterizes BS-rigidity and stability among finitely-generated amenable groups.
\begin{theorem}[BS-rigidity and stability of amenable groups]\label{thm:PositiveAmenable}
Let $\Gamma$ be a finitely-generated amenable group. Then,
\begin{enumerate}
\item\label{enum:PostiveAmenableBS-rigid} \cite{BLM} $\Gamma$ is BS-rigid.
\item \label{enum:PostiveAmenableStable}\cite{BLT} $\Gamma$ is stable if and only if every
invariant random subgroup of $\Gamma$ is co-sofic, i.e., a weak-$\ast$ limit of invariant random subgroups supported on finite-index subgroups.
\end{enumerate}
\end{theorem}
In \cref{sec:IRS} we recall the relevant terminology regarding invariant random subgroups.
The proof of \cref{thm:PositiveAmenable} relies on a theorem of Ornstein--Weiss \cite{OrnsteinWeiss} (see also \cite{ConnesFeldmanWeiss}) and a theorem of Newman--Sohler \cite{NewmanSohler2013} (see also \cite{Elek2012}). \cref{prop:DistingusihableStableEquivUndercoSoficCondition} below explains how the first part of \cref{thm:PositiveAmenable} implies one direction of the if and only if statement in the second part.

\cref{thm:PositiveAmenable}(2) implies that every finitely-generated nilpotent (and even polycyclic) group is stable, and the same is true for the solvable Baumslag--Solitar groups $\BS(1,n)$, $n\in\N$ (see \cite[Theorem 1.2]{BLT}).
This includes, as a special case, the stability of $E=\inset{\varX\varY\varX^{-1}\varY^{-1}}$ (originally proved in \cite{ArzhantsevaPaunescu}).
 \cref{thm:PositiveAmenable} also enables us to find groups that are BS-rigid but not stable (and subsets of $F_S$ which are testable but not stable), see \cref{sec:examples}.
 
Moving on from amenable groups to the groups that are as far as possible from being amenable, \cref{thm:NegativeKazhdan} below is a negative result about BS-rigidity and stability of groups with Kazhdan's property $\T$. In fact, property $\ptau$ (see \cite{LubotzkyExpandersBook}) suffices.
\begin{theorem}[Negative results for groups with property $\ptau$]\label{thm:NegativeKazhdan}
Let $\Gamma$ be a finitely-generated group with property
$\ptau$. Then,
\begin{enumerate}
\item {\cite{BLM}} If $\Gamma$ has infinitely many finite-index subgroups,
then $\Gamma$ is not BS-rigid.
\item \cite{BeckerLubotzky} If $\Gamma$ is infinite and sofic, then it is not stable.
\end{enumerate}
\end{theorem}

In \cref{sec:ProofNegativeKazhdan} we prove \cref{thm:NegativeKazhdan}(1) using the construction of \cite{BLM} (and \cite{BeckerLubotzky}), but with a somewhat different argument via the language of invariant random subgroups.

\subsection{Stability and testability in the flexible model}\label{sec:flexModel}

\deffont{Flexible stability} and \deffont{flexible BS-rigidity} are important relaxations of stability and BS-rigidity, respectively. Flexible stability in permutations was originally defined in \cite{BeckerLubotzky} (cf.  \cite{GowersHatami,DOT}). Flexible BS-rigidity is a new concept, hinted at in \cite{BLM}, and formulated here. In order to define these notions, we need to define the distance between permutations on sets of different cardinalities, extending  \eqref{eq:PermutationMetric}: Given $\sigma \in \Sym(n)$ and $\tau \in \Sym(N)$ with $N\ge n$, let
$$
d(\sigma,\tau)=d(\tau,\sigma) = \frac{\inabset{x\in [n]\mid \sigma x\ne \tau x}}{n}\eperiod
$$
Fix a function $\nu:\R_{>0}\times \N\to \N\cup \{0,\infty\}$, which is monotone non-decreasing in the first argument. For a subset $E$ of $F_S$, let
$$
\Sol_E^{<\eps, \nu\flex}(n) = \inset{\ol\sigma\in \Sym(n)^k\mid \exists N\in \N~\exists \ol\tau \in \Sol_E(N), ~n\le N\le n+\nu(\eps,n)\text{ and } \sum_{i=1}^k d(\sigma_i,\tau_i)<\eps}
$$
and
$$\Sol_E^{\ge \eps, \nu\flex}(n) = \Sol_E(n) \setminus \Sol_E^{<\eps,\nu\flex}(n)\eperiod$$
Also,
$\Sol_E^{<\eps, \nu\flex} = \bigcup_{n\in \N} \Sol_E^{<\eps, \nu\flex}(n)$ and $\Sol_E^{\ge\eps, \nu\flex} = \bigcup_{n\in \N} \Sol_E^{\ge\eps, \nu\flex}(n)$.

\deffont{Flexible BS-rigidity} and \deffont{flexible stability} are defined analogously to the respective non-flexible notions:
\begin{definition}[Flexible stability and testability]\label{def:flexModel}
Fix a finite subset $E$ of $F_S$.
\begin{enumerate}
    \item $E$ is \deffont{$\nu$-flexibly testable} if for every $\eps > 0$, the pair $\inparen{\Sol_E,\Sol_E^{\ge \eps, \nu\flex}}$ is separable. 
    \item $E$ is \deffont{$\nu$-flexibly stable} if for every $\eps > 0$ there exists $s\in \N$ such that \SAS{} with word set $E$ and repetition factor $s$ is a separator for $\inparen{\Sol_E,\Sol_E^{\ge \eps, \nu\flex}}$.
\end{enumerate}
\end{definition}

Note that $E$ is testable (resp.\ stable) if and only if it is $\nu$-flexibly testable (resp. $\nu$-flexibly stable) for $\nu = 0$.

\begin{definition}\label{def:flexModelSpecifics}
Let $E\subseteq F_S$ and suppose that $E$ is $\nu$-flexibly testable (resp.\ $\nu$-flexibly stable).
\begin{enumerate}
    \item $E$ is \deffont{$O(\eps n)$-flexibly testable} (resp. \deffont{$O(\eps n)$-flexibly stable})\footnote{The notion $O(\eps n)$-flexible testability remains the same if one replaces $\eps$ by any function of $\eps$ that goes to $0$ as $\eps\to 0$. Indeed, by unwinding the definitions, one can see, for example, that $E$ is $O(\eps n)$-flexibly testable if and only if it is $O(\eps^2 n)$-flexibly testable, i.e., $\nu(\eps, n)$-flexibly testable where $\nu(\eps,n)\le c\eps^2 n$. The same holds for the notions of flexible stability and flexible BS-rigidity.}
    if $\nu(\eps,n) \le c\eps n$ for some fixed $c$ for all $\eps>0$ and $n\in \N$.

    \item $E$ is \deffont{$O(n)$-flexibly testable} (resp. \deffont{$O(n)$-flexibly stable})
    if $\nu(\eps,n) \le cn$ for some fixed $c$ for all $\eps>0$ and $n\in \N$.
    
    \item $E$ is \deffont{$\UB$-flexibly testable} (resp. \deffont{$\UB$-flexibly stable})
    if $\nu(\eps,n)=\infty$ for all $\eps>0$ and $n\in \N$
    (here $\UB$ stands for \emph{unboundedly}).
\end{enumerate}    
\end{definition}

In the existing literature, $O(\eps n)$-flexible stability is simply called flexible stability \cite{BeckerLubotzky,LLM}, while $\UB$-flexible stability is called \deffont{very-flexible stability} \cite{BeckerLubotzky,Ioana}. The terminology presented in \cref{def:flexModelSpecifics} is intended to be more informative.

\begin{definition}[Flexible BS-rigidity]
	A set $E\subseteq F_S$ is \deffont{$\nu$-flexibly BS-rigid} if for every $\eps > 0$ there exist a finite set $P = P(\eps)\subseteq F_S$ and $\delta = \delta(\eps) > 0$ such that $\dTV\inparen{N_{\ol \sigma,P}, N_{\ol \tau,P}} \ge \delta$ for every $n\in N$, $\ol \sigma \in \Sol_{E}^{\ge \eps,\nu\flex}(n)$ and $\ol \tau \in \Sol_{E}(n)$.
	
	The notions \deffont{$\UB$-flexible BS-rigidity}, \deffont{$O(\eps n)$-BS-rigidity} and \deffont{$O(n)$-BS-rigidity} are defined analogously to the respective notions in \cref{def:flexModelSpecifics}.
\end{definition}

\cref{prop:BS-rigidViaHoms} can be generalized to the flexible model:
\begin{prop}\label{prop:BS-rigidViaHomsFlex}
Let $E\subseteq F_S$ and write $\Gamma=\langle S \mid E \rangle$. The following conditions are equivalent.
\begin{enumerate}
    \item $E$ is $\nu$-flexibly BS-rigid.
    \item\label[condition]{enum:BS-rigidityByHomsFlex}
    For every sequence of functions $\inparen{f_\ell\colon\Gamma\to \Sym(n_\ell)}_{\ell=1}^\infty$, $n_\ell\to\infty$, such that
    $$d\inparen{f_\ell(\gamma_1\gamma_2),f_\ell(\gamma_1)f_\ell(\gamma_2)}\to_{\ell\to \infty} 0 \qquad\forall \gamma_1,\gamma_2\in \Gamma\ecomma$$
    if there is a sequence of homomorphisms $\inparen{g_\ell\colon\Gamma\to \Sym(n_\ell)}_{\ell=1}^\infty$, 
    $$d_{\TV}(N_{f_\ell,Q},N_{g_\ell,Q})\to_{\ell\to\infty} 0\qquad\forall Q\subset\Gamma, |Q|<\infty$$
    then there is a sequence $\inparen{\eps_\ell}_{\ell=1}^\infty$, $\eps_\ell \to_{\ell\to \infty} 0$ and a sequence of homomorphisms $\inparen{h_\ell\colon\Gamma\to \Sym(N_\ell)}_{\ell=1}^\infty$, ($N_\ell \in \N$) with
    $$n_\ell \le N_\ell \le n_\ell+\nu(\eps_\ell,n_\ell)$$
    such that 
    $$d\inparen{f_{\ell}(\gamma),h_{\ell}(\gamma)} \to_{\ell\to \infty} 0 \qquad \forall \gamma\in  \Gamma\eperiod$$
\end{enumerate}
\end{prop}
\cref{prop:BS-rigidViaHomsFlex} is proved in \cref{sec:ProofBS-rigidViaHoms}.

\begin{definition} Let $\Gamma$ be a group.
\begin{enumerate}
    \item $\Gamma$ is \deffont{$\nu$-flexibly BS-rigid} if it satisfies \cref{enum:BS-rigidityByHomsFlex} of \cref{prop:BS-rigidViaHomsFlex}.
    \item $\Gamma$ is \deffont{$\nu$-flexibly stable} if for every sequence $(f_\ell\colon\Gamma\to\Sym(n_\ell))_{\ell=1}^\infty$, $n_\ell\to\infty$, satisfying \eqref{eq:AsymptoticHomomorphism} (see \cref{prop:stability-defs}), there is a sequence of homomorphisms $(h_\ell\colon\Gamma\to\Sym(N_\ell)$, $n_\ell\leq N_\ell\leq n_\ell+\nu(\eps,n_\ell)$, satisfying \eqref{eq:AsymptoticallyCloseFuncs}.
\end{enumerate}
    
\end{definition}

One can prove that a finite subset $E$ of $F_S$ is $\nu$-stable if and only if the same is true for the group $\inang{S\mid E}$ (we omit the proof, as it is very similar to the proof of \cref{prop:BS-rigidViaHomsFlex}). Using this and \cref{prop:BS-rigidViaHomsFlex}, we deduce the following generalization of \cref{prop:GroupProperty}.

\begin{corollary}\label{cor:GroupPropertyFlex}
Fix finite sets $S_1$ and $S_2$,
and take sets $E_{1}\subset F_{S_1}$ and $E_{2}\subset F_{S_2}$ 
such that the groups $\langle S_1 \mid E_1 \rangle$ and $\langle S_2 \mid E_2 \rangle$ are isomorphic. 
Then, $E_{1}$ is $\nu$-flexibly BS-rigid (resp. $\nu$-flexibly stable) if and only if the same is true for $E_{2}$.
\end{corollary}

The \LSM{} algorithm and \cref{thm:LSMUniversal} also generalize immediately to the flexible setting\footnote{Indeed, in the proof of \cref{thm:LSMUniversal} (see \cite[Thm.\ 3]{BLM}) one can take, in place of $\Sol_E$ and $\Sol^{\ge \eps}_E$, any  two sets $A,B \in \bigcup_{n\in \N}\Sym(n)^k$ that are invariant under relabeling of the base set $[n]$.}. Hence, a finite subset $E$ of $F_S$ is $\nu$-flexibly testable if and only if it is $\nu$-flexibly BS-rigid.

It is worth noting that for a finitely-generated amenable group, stability and $O(\eps n)$-flexible stability are equivalent. To see this, in light of \cref{thm:PositiveAmenable}(\ref{enum:PostiveAmenableStable}), it suffices to prove that if an amenable group $\Gamma$ is $O(\eps n)$-flexibly stable then every invariant random subgroup of $\Gamma$ is co-sofic. The proof of the last claim is identical to the proof of \cite[Theorem 7.10]{BLT}. We state this as a corollary:
\begin{corollary}
\label{cor:amenable-stab-iff-flexistab}
Let $\Gamma$ be a finitely-generated amenable group. Then $\Gamma$ is stable if and only if it is $O(\eps n)$-flexibly stable.
\end{corollary}
The amenability assumption in \cref{cor:amenable-stab-iff-flexistab} is probably essential. In particular, one may hope to find a $\nu$-flexibly stable group with Kazhdan's property $\T$ and infinitely many finite quotients (for some $\nu$).
Such a group is not stable in the strict sense (i.e., that of \cref{def:stability}, namely with $\nu=0$) by \cref{thm:NegativeKazhdan}.
Interestingly, by \cite[Lemma 3.2(2)]{Ioana}, a countable group with property $\ptau$ is UB-flexibly stable if and only if it is $\nu$-flexibly stable for some $\nu$ satisfying $\nu(\eps,n)=o(n)$.
It is especially interesting to know whether $\PSL_d(\ZZ)$ is UB-flexibly stable for at least one $d\geq 5$ because, by \cite{BowenBurton}, this will imply the existence of a non-sofic group.

\subsection{General properties}\label{sec:GeneralProps}

Here we describe some general properties of stability and BS-rigidity.  Propositions \ref{prop:DistinguishabilityByFiniteResidual},  \ref{prop:quotient-fg}, \ref{prop:FiniteIndexSubgroup} and \ref{prop:RF-not-weakly-stable} are proved in \cref{sec:ProofsGeneralProps}, while \cref{prop:DistingusihableStableEquivUndercoSoficCondition} is proved in \cref{sec:IRSFormulation}.

For a group $\Gamma$, write $R\left(\Gamma\right)$ for the finite
residual of $\Gamma$, i.e., the intersection of all finite-index
subgroups of $\Gamma$. Fix a function $\nu:\R_{>0}\times \N\to \N\cup \{0,\infty\}$, monotone non-decreasing in the first argument.

\begin{prop}\label{prop:DistinguishabilityByFiniteResidual}
A finitely-generated group $\Gamma$ is BS-rigid (resp.\ $\nu$-flexibly BS-rigid) if and only if the same is true for $\Gamma/R\left(\Gamma\right)$.
\end{prop}

\begin{prop}
\label{prop:quotient-fg}
Let $N$ be a normal subgroup of a finitely-generated group $\Gamma$, and suppose that $N$ is finitely-generated as a group. Then,
\begin{enumerate}
    \item If $\Gamma$ is stable (resp.\ $\nu$-flexibly stable) then the same is true for $\Gamma/N$.
    \item If $\Gamma$ is BS-rigid (resp.\ $\nu$-flexibly BS-rigid) then the same is true for $\Gamma/N$.
\end{enumerate}
\end{prop}

\begin{prop} \label{prop:FiniteIndexSubgroup}
Let $\Gamma$ be a $\nu$-flexibly stable (resp. $\nu$-flexibly BS-rigid) group, and let $H$ be a finite-index subgroup of $\Gamma$. Then $H$ is $\nu'$-flexibly stable (resp. $\nu'$-flexibly BS-rigid) for $\nu'(\eps,n)=([\Gamma:H]-1)n+\nu(\eps,[\Gamma:H]n)$.
\end{prop}

In \cref{prop:RF-not-weakly-stable,prop:DistingusihableStableEquivUndercoSoficCondition} below, we state additional connections between stability and BS-rigidity. \cref{prop:RF-not-weakly-stable} is concerned with a certain relaxation of stability, called \deffont{weak stability}:

\begin{definition}[Weak stability]
A group $\Gamma$ is \deffont{$\nu$-flexibly weakly stable} (or simply \deffont{weakly stable} when $\nu = 0$) if
for every sequence of functions $\inparen{f_\ell:\Gamma\to \Sym(n_\ell)}_{\ell=1}^\infty$, $n_\ell\to\infty$, such that
        \begin{equation}
        \label{eq:weak-almost-hom}
        \inparen{f_\ell(\gamma_1\gamma_2),f_\ell(\gamma_1)f_\ell(\gamma_2)}\to_{\ell\to \infty} 0\qquad\forall \gamma_1,\gamma_2\in \Gamma\ecomma
                \end{equation}
        and
        \begin{equation}
        \label{eq:weak-almost-injective}
        d\inparen{f_\ell(\gamma),1_{\Sym(n_\ell)}}\to_{\ell\to\infty} 1\qquad\forall \gamma\in \Gamma\setminus{\inset{1_\Gamma}}\eperiod
        \end{equation}
        there exists a sequence of homomorphisms $\inparen{h_\ell:\Gamma\to \Sym(N_\ell)}_{\ell=1}^\infty$, $n_\ell \le N_\ell \le n_\ell+\nu(\eps_\ell,n_\ell)$, such that
        $$d\inparen{f_\ell(\gamma),h_\ell(\gamma)}\to_{\ell\to\infty} 0\qquad\forall \gamma\in \Gamma\eperiod$$
\end{definition}

\begin{proposition}
\label{prop:RF-not-weakly-stable}
A finitely-generated residually-finite BS-rigid (resp.\ $\nu$-flexibly BS-rigid) group is weakly stable (resp.\ $\nu$-flexibly weakly stable).
\end{proposition}
In light of \cref{thm:PositiveAmenable}(1),
\cref{prop:RF-not-weakly-stable} explains and generalizes \cite[Theorem 1.1]{ArzhantsevaPaunescu}, which states that residually-finite amenable groups are weakly stable.

\begin{proposition}\label{prop:DistingusihableStableEquivUndercoSoficCondition}
\label{prop:if-dense-testable-is-stable}If every invariant random subgroup of $\Gamma$
is co-sofic then $\Gamma$ is BS-rigid (resp.\ $\nu$-flexibly BS-rigid) if and only if $\Gamma$ is
stable (resp.\ $\nu$-flexibly stable).
\end{proposition}

\cref{prop:DistingusihableStableEquivUndercoSoficCondition} sheds light on a connection between the two parts of \cref{thm:PositiveAmenable}. Specifically, \cref{thm:PositiveAmenable}(\ref{enum:PostiveAmenableBS-rigid}), together with \cref{prop:DistingusihableStableEquivUndercoSoficCondition}, implies the positive part of \cref{thm:PositiveAmenable}(\ref{enum:PostiveAmenableStable}), namely, that a finitely-generated amenable group whose invariant random subgroups are all co-sofic is stable.

\subsection{Examples}\label{sec:examples}
We now use the toolkit we have built to state some positive and negative results about stability and BS-rigidity of concrete groups.
Below we write $F_m$ for a free group with a basis of cardinality $m$.
\begin{prop}
\label{prop:virtually-ioana-not-weakly}
Let $\Gamma$ be a group satisfying at least one of the following conditions.
\begin{enumerate}
    \item\label[condition]{cond:Ioana1} $\Gamma$ contains a finite-index subgroup $H$ such that $H\cong F_{m}\times F_{n}$ ($m\geq2$,
$n\geq2$).
    \item\label[condition]{cond:Ioana2} $\Gamma$ contains a finite-index subgroup $H$ such that $H\cong F_{m}\times\Z^{d}$ ($m\geq2$, $d\geq1$).
\end{enumerate}
Then, $\Gamma$ is not $\UB$-flexibly BS-rigid.
\end{prop}
\begin{proof}
The group $\Gamma$ is clearly residually finite. On the other hand,
\cite[Thm.\ D]{Ioana} states that a group satisfying either of \cref{cond:Ioana1} or \cref{cond:Ioana2} above is not UB-flexibly weakly stable. Thus, $\Gamma$
is not $\UB$-flexibly BS-rigid by Proposition \ref{prop:RF-not-weakly-stable}.
\end{proof}

\begin{prop}
\label{prop:BS}
Let $\Gamma$ denote the Baumslag--Solitar group $\BS\left(m,n\right)=\inang{\varX,\varY\mid \varX\varY^m=\varY^n\varX}$.
\begin{enumerate}
    \item If $m$ and $n$ are co-prime, then $\Gamma$ is BS-rigid.
    \item If $|m|=|n|\geq 2$ then $\Gamma$ is not UB-flexibly BS-rigid.
\end{enumerate}
\end{prop}

\begin{proof}
(1) When $m$ and $n$ are co-prime, the quotient $\Gamma/R\left(\Gamma\right)$ is isomorphic to the metabelian
group $\ZZ\left[\frac{1}{mn}\right]\rtimes_{\theta}\ZZ$, where $\theta(1)$ is the automorphism of $\ZZ\left[\frac{1}{mn}\right]$ given by multiplication by $\frac{m}{n}$.
In particular, $\Gamma/R\left(\Gamma\right)$ is amenable, and thus
$\Gamma$ is BS-rigid by Proposition \ref{prop:DistinguishabilityByFiniteResidual} and Theorem \ref{thm:PositiveAmenable}(1).

(2) When $|m|=|n|\geq 2$, the group $\Gamma$ has a finite-index subgroup isomorphic to $F_k\times\ZZ$ for some $k\geq 2$ (see., e.g., \cite[Proposition 2.6]{Levitt05}), and thus $\Gamma$ is not UB-flexibly BS-rigid by \cref{prop:virtually-ioana-not-weakly}.
\end{proof}
It is interesting to note that the proof of \cref{prop:BS}(1) shows that the single equation $\mathsf{XY^m=Y^n X}$ ($m,n$ co-prime) is testable by using the BS-rigidity of the non-finitely-presented group $\Gamma/R(\Gamma)$.

The question of whether $\BS(m,n)$ is (flexibly) BS-rigid in the cases not covered by \cref{prop:BS} is still open. On the other hand, it is known that $\BS(m,n)$ is stable if $|m|\leq 1$ or $|n|\leq 1$ \cite[Theorem 1.2(ii)]{BLM}, and not UB-flexibly stable otherwise \cite[Example 7.3]{ArzhantsevaPaunescu}, \cite[Corollary B]{Ioana}.

\begin{prop}
For $n\geq 3$ the braid group $B_n$ and pure braid group $PB_n$ are not UB-flexibly BS-rigid.
\end{prop}
\begin{proof}
The group $PB_3$ is isomorphic to $F_2\times \ZZ$, and thus, by \cref{prop:virtually-ioana-not-weakly}, it is not UB-flexibly rigid.
Now, $PB_m$ surjects onto $PB_{m-1}$ with a finitely-generated kernel for all $m\geq 3$. Thus, for $m\geq 3$, repeated applications of \cref{prop:quotient-fg} imply that $PB_m$ is not UB-flexibly BS-rigid.
Finally, for $m \geq 3$, $B_m$ is not UB-flexibly BS-rigid since $PB_m$ is a finite-index subgroup of $B_m$ and by \cref{prop:FiniteIndexSubgroup}.
\end{proof}
It is interesting to note that the $(2,3)$-torus knot group $\inang{\varX,\varY \mid \varX^2=\varY^3}$ is not UB-flexibly BS-rigid because it is isomorphic to the braid group $B_3$. It would be interesting to study the stability and BS-rigidity of torus knot groups further.

Next we show that $\Aut\left(F_{n}\right)$, $\Out\left(F_{n}\right)$ ($n\ge 3)$ and $\MCG\left(g\right)$ ($g\ge 3$)
are not BS-rigid (but note that the questions of their flexible BS-rigidity
remains open). Note that $\Out(F_2)$ is BS-rigid, and even stable, since, by \cite{LazarovichLevit}, virtually-free groups are stable.

\begin{prop}~
\begin{enumerate}
    \item For $n\geq3$, the groups $\Aut\left(F_{n}\right)$ and $\Out\left(F_{n}\right)$ are not BS-rigid.
    \item For $g\geq3$, the mapping class group $\MCG\left(g\right)$ is not
BS-rigid.
\end{enumerate}
\end{prop}

\begin{proof}~
\begin{enumerate}
    \item The groups $\Aut\left(F_{n}\right)$ and $\Out\left(F_{n}\right)$ surject onto $\GL_n(\ZZ)$ with a finitely-generated kernel \cite{Magnus}. But $\GL_n(\ZZ)$ is not BS-rigid by \cref{thm:NegativeKazhdan}(1) because it is infinite, residually finite, and has property $\T$. Thus $\Aut(F_n)$ and $\Out(F_n)$ are not BS-rigid by \cref{prop:quotient-fg} and \cref{thm:NegativeKazhdan}.
    \item The group $\MCG(g)$ surjects onto $\Sp_{2g}(\ZZ)$
        with a kernel $N$, known as the Torelli subgroup, which is a finitely-generated group
        \cite{Johnson_Torelli}. But $\Sp_{2g}(\ZZ)$
        is not BS-rigid because it is infinite, residually finite, and has
        property $\T$. Thus $\MCG(g)$ is not BS-rigid by \cref{prop:quotient-fg} and \cref{thm:NegativeKazhdan}.
\end{enumerate}
\end{proof}

\section{Suggestions for further research}
Here we discuss interesting questions for further study, including some problems that appear in \cite[Sec.\ 2.2.]{BLM}, as well as new directions.

\subsection{Query complexity}\label{sec:QueryComp}
Let $E$ be a finite BS-rigid subset of $F_S$. Then $E$ is testable.
Hence, for each $\eps > 0$, there is a $\inparen{\Sol_E,\Sol_E^{\ge \eps}}$-separator $\cM_\eps$. From an algorithmic perspective, it is desirable to minimize the query complexity of $\cM_\eps$ as a function of $\eps$ (see also \cite[Sec.\ 2.1.4]{BLM}).

In \cite{BeckerMosheiff}, it is shown that if the group $\Gamma = \inang{S\mid E}$ is abelian then $\Gamma$ is \deffont{polynomially stable}. Namely, there exists $D=D(\Gamma) \ge 1$ such that $\inparen{\Sol_E,\Sol_E^{\ge \eps}}$ is separable by a \SAS{} algorithm with query complexity $O_E\inparen{\inparen{\frac 1\eps}^D}$. Finding the optimal $D$ (known as the \deffont{degree of polynomial stability}) for a given abelian group $\Gamma$ is an open problem. In particular, even the following problem is open.
\begin{problem}
Determine the degree of polynomial stability of $\Z^2$.
\end{problem}
In \cite{BeckerMosheiff}, it is shown that, for $k\ge 2$, the degree of polynomial stability of $\Z^k$, denoted $D_k$, satisfies $k\le D_k\le 2^{O(k)}$. In the $O(\eps n$)-flexible model, the best known bounds are $2 \le D_k \le 2^{O(k)}$.

The following basic problems are also open.
\begin{problem}
Is there a non-abelian polynomially-stable group?
\end{problem}
\begin{problem}
Is there a stable group that is not polynomially stable?
\end{problem}

The flexible variants of each of the above problems are also interesting and open. It it worth noting that in the context of stability in unitary groups, the cohomological method of \cite{DGLT} implies stability with linear rate whenever it applies (see also \cite{MoralesGlebsky}).

In the realm of flexible stability, the main result of \cite{LLM} implies that when $\Gamma = \inang{S\mid E}$ is a surface group of genus at least $2$, there is some  $\nu(\eps,n) \le O(\eps n)$ such that $\inparen{\Sol_E,\Sol_E^{\ge \eps,\nu\flex}}$ is separable by a \SAS{} separator with query complexity $O\inparen{\frac{1}{\eps} \log\frac{1}{\eps}}$.

It is interesting to bound the query complexity of other testable subsets of $F_S$ as well, including those that are not stable. The following conjecture appears in \cite{BLM}.
\begin{conjecture}
Suppose that $\Gamma = \inang{S\mid E}$ is a finitely-presented amenable group, and thus $E$ is testable by \cref{thm:PositiveAmenable,thm:LSMUniversal}. Then, $E$ is testable with query complexity bounded from above by some function of $|S|$, $\sum_{w\in E}|w|$, and the F\o lner function of $\Gamma$.
\end{conjecture}

\subsection{Time complexity}\label{sec:TimeComp}
The time complexity of \SAS{} is proportional to its query complexity. In particular, for a stable group, the \SAS{} algorithm provides separators whose time complexity does not depend on the size $n$ of the input permutations.

The case for a BS-rigid instable subset $E$ of $F_S$ might be different. The time complexity of \LSM{} might be higher than its query complexity because of the computation involved the evaluation of \eqref{eq:LSMCondition} in \cref{alg:LSM}. 
\begin{problem}
Bound the time complexity for various instable BS-rigid finite subsets $E\subseteq F_S$. In particular, it is interesting to check, for various such sets $E$, whether the time complexity of the separators can be made independent on the size $n$ of the input permutations.
\end{problem}

\subsection{Allowing the number of queries to depend on $n$}
Consider a relaxed version of algorithmic separation (\cref{def:separation}) in which the maximum number of queries made by the separating algorithm is $q(\eps,n)$ for some $q:\R_{>0}\times \N\to \N$. Allowing $q(n) \ge n$ is not interesting, since the input can then be read in its entirety, and so every two disjoint subsets of $\Sym(n)^k$ can be separated. A more modest dependence on $n$ is interesting, and may result in natural, weaker but non-trivial notions of testability. In particular, consider the following problem.

\begin{problem}
Let $E\subseteq F_S$ be a finite set such that $\Gamma = \inang{S\mid E}$ has property $\ptau$ and infinitely many finite-index subgroups (by \cref{thm:LSMUniversal,thm:NegativeKazhdan}, $E$ is not testable in the usual sense). Are there examples where $E$ is testable with $q(\eps,n)$ queries for some function $q$ with small dependence on $n$? Asymptotically, how small can we make this dependence?
\end{problem}

\subsection{Fixed radius BS-rigidity}
Fixed radius BS-rigidity is a stricter variant of BS-rigidity (\cref{def:BS-rigid}). In this stricter definition, the set $P$ must be fixed, and may not depend on $\eps$.
\begin{definition}
	A subset $E$ of $F_S$ is \deffont{fixed radius BS-rigid} (or \deffont{FR-BS-rigid} for short) if there exists a finite set $P\subseteq F_S$ such that for every $\eps > 0$ there is $\delta = \delta(\eps) > 0$ satisfying $\dTV\inparen{N_{\ol \sigma,P}, N_{\ol \tau,P}} \ge \delta$ for every $n\in \N$, $\ol \sigma \in \Sol_{E}^{\ge \eps}(n)$ and $\ol \tau \in \Sol_{E}(n)$. In this case we also say that $E$ is $\inparen{P,\delta(\eps)}$-FR-BS-rigid.
\end{definition}

FR-BS-rigidity is a natural notion, which also has an algorithmic motivation: If E is FR-BS-rigid then it is testable by family of \LSM{} separators with $P$ constant.
This is a weaker variant of \deffont{proximity oblivious testability} (see \cite[Def.\ 1.7]{Goldreich}). More details about this connection are given in \cite[Sec.\ 2.2.6]{BLM}.

Note that every stable subset $E$ of $F_S$ is $(E,\delta(\eps))$-FR-BS-rigid for some function $\delta\colon \RR_{>0}\to\RR_{>0}$.\begin{problem}
Is there an instable FR-BS-rigid finitely-generated group?
\end{problem}

\section{Stability, BS-rigidity and invariant random subgroups}\label{sec:IRS}
Let $\Gamma$ be a finitely-generated group.
Here we recall relevant definitions pertaining to invariant random subgroups, and prove equivalent criteria for stability and BS-rigidity. We consider the space $\Sub(\Gamma)$ of all subgroups of $\Gamma$ with the topology induced from the inclusion of $\Sub(\Gamma)$ in $\{0,1\}^\Gamma$, where the latter is endowed with the product topology, and $\inset{0,1}$ with the discrete topology.
For finite subsets $A$ and $B$ of $\Gamma$, the set $C_{A,B}=\{H\in\Sub(\Gamma)\mid H\cap A=B\}$ is clopen in $\Sub(\Gamma)$. The collection of all sets $C_{A,B}$ of this form is a basis for the topology on $\Sub(\Gamma)$.
An \deffont{invariant random subgroup} of $\Gamma$ is a Borel regular probability measure $\mu$ on $\Sub(\Gamma)$ such that $\mu(\gamma A \gamma^{-1})=\mu(A)$ for every Borel subset $A$ of $\Sub(\Gamma)$ and $\gamma\in\Gamma$. The space of all invariant random subgroups of $\Gamma$ is denoted by $\IRS(\Gamma)$. This space is equipped with the weak-$\ast$ topology, i.e., a sequence $(\mu_n)_{n=1}^\infty$ of measures in $\IRS(\Gamma)$ converges to $\mu\in\IRS(\Gamma)$ if and only if $\lim_{n\to\infty}\mu_n(C_{A,B})=\mu(C_{A,B})$ for all finite subsets $A$ and $B$ of $\Sub(\Gamma)$.

An action of $\Gamma$ on a finite set $X$ gives rise to an invariant random subgroup $\mu_X$ given by $\mu_X(A)=\PP(\stab_\Gamma(x)\in A)$, where $x$ is sampled uniformly from $X$.
The measure $\mu_X$ is called the \deffont{random stabilizer} of $X$.
We say that a sequence $(X_n)_{n=1}^\infty$ is \deffont{convergent} if the sequence of measures $\mu_{X_n}$ converges to some measure $\mu$. In this case, we also say that the sequence $(X_n)_{n=1}^\infty$ \deffont{converges} to $\mu$ and write $X_n\to\mu$. We write $\IRS_{\findex}(\Gamma)$ for the space of invariant random subgroups of $\Gamma$ supported on finite-index subgroups. Then $\mu_{X_n}$ as above belongs to $\IRS_{\findex}(\Gamma)$. An invariant random subgroup $\mu\in\IRS(\Gamma)$ is \deffont{co-sofic} if it is the limit of a sequence of elements of $\IRS_{\findex}(\Gamma)$.
In this context, we should mention the Aldous--Lyon conjecture, which states that every invariant random subgroup of a finitely-generated free group is co-sofic \cite{AldousLyons}.

Fix an epimorphism $\pi\colon F_S\to\Gamma$, where $F_S$ is a finitely-generated free group. The embedding of $\Sub(\Gamma)$ in $\Sub(F_S)$, sending $H$ to $\pi^{-1}(H)$, gives rise to an embedding of $\IRS(\Gamma)$ in $\IRS(F_S)$. Henceforth, we shall view $\IRS(\Gamma)$ as closed subspace of $\IRS(F_S)$.
Note that it is possible for an invariant random subgroup $\mu\in\IRS(\Gamma)$ to be co-sofic in $F_S$ but not in $\Gamma$, i.e., $\mu$ may be a limit of elements of $\IRS_{\findex}(F_S)$, but not a limit of elements of $\IRS_{\findex}(\Gamma)$.

\subsection{Stability and BS-rigidity in terms of invariant random subgroups}\label{sec:IRSFormulation}
Here we reformulate the notions of stability and BS-rigidity in terms of invariant random subgroups (\cref{prop:IRSFormulation}).

Given a group $\Gamma$, a \deffont{$\Gamma$-set} is a set equipped with an action of $\Gamma$. For two finite $F_S$-sets $X$ and $Y$ with $|X|\le |Y|$, let
\begin{equation}\label{eq:FSetDistance}
    d_S(Y,X) = d_S(X,Y)=\min_{f\colon X\to Y}\frac{1}{|X|}\sum_{s\in S}\sum_{x\in X}{\bf 1}_{f(sx)\neq sf(x)}\ecomma
\end{equation}
where the minimum is taken over all injections $f$ from $X$ to $Y$.

\begin{defn}
\label{def:of-all-defs}Let $(X_{n})_{n=1}^\infty$
be a convergent sequence of finite $F_S$-sets, $X_n\to\mu\in\IRS(F_S)$. Then:
\begin{enumerate}
\item $(X_{n})_{n=1}^{\infty}$ is a \deffont{stability challenge}
for $\Gamma$ if $\mu\in\IRS\left(\Gamma\right)$.
\item $(X_{n})_{n=1}^{\infty}$ is a \deffont{BS-rigidity challenge}
for $\Gamma$ if $\mu\in\IRS(\Gamma)$ and $\mu$ is co-sofic (i.e., co-sofic in $\Gamma$).
\item $(X_{n})_{n=1}^{\infty}$ is a \deffont{sofic approximation}
for $\Gamma$ if $\mu\in\IRS(\Gamma)$ and $\mu$ is the Dirac measure concentrated on the trivial subgroup of $\Gamma$.
\item Two sequences $(X_{n})_{n=1}^\infty$ and $(Y_{n})_{n=1}^\infty$
of finite $F_S$-sets are \deffont{equivalent} if $d_{S}\left(X_{n},Y_{n}\right)\to0$.
\item A sequence of finite $\Gamma$-sets $(Y_{n})_{n=1}^{\infty}$
is a \deffont{$\nu$-flexible solution} (or simply \deffont{solution} when $\nu = 0$) for $(X_{n})_{n=1}^{\infty}$
if it is equivalent to $(X_{n})_{n=1}^\infty$, and there exists a sequence $(\eps_n)_{n=1}^\infty$, $\eps_n \to 0$ such that $|X_n|\le |Y_n|\le |X_n|+\nu(\eps_n,|X_n|)$ for all $n$. 
\end{enumerate}
\end{defn}

\begin{prop}\label{prop:IRSFormulation}
Let $\Gamma$ be a finitely-generated group.
\begin{enumerate}
    \item $\Gamma$ is stable (resp.\ $\nu$-flexibly stable) if and only if every stability challenge for $\Gamma$ has a solution (resp.\ $\nu$-flexible solution).
    \item $\Gamma$ is BS-rigid (resp.\ $\nu$-flexibly BS-rigid) if and only if every BS-rigidity challenge
    for $\Gamma$ has a solution (resp.\ $\nu$-flexible solution).
    \item $\Gamma$ is weakly stable (resp.\ $\nu$-flexibly weakly stable) if and only if every sofic approximation for $\Gamma$ has a solution (resp.\ $\nu$-flexible solution).
\end{enumerate}
\end{prop}
We prove \cref{prop:IRSFormulation} in \cref{sec:ProofIRSFormulation}.

\cref{prop:DistingusihableStableEquivUndercoSoficCondition} follows immediately from \cref{prop:IRSFormulation}. 
\begin{proof}[Proof of \cref{prop:DistingusihableStableEquivUndercoSoficCondition}]
Clearly, if $\Gamma$ is stable then it is also BS-rigid. For the other direction, suppose that $\Gamma$ is BS-rigid and that every IRS of $\Gamma$ is co-sofic. Let $\inparen{X_n}_{n=1}^{\infty}$ be a stability challenge for $\Gamma$, with $X_n\to \mu$ for some $\mu\in \IRS(\Gamma)$. By assumption, $\mu$ is co-sofic, so $\inparen{X_n}_{n=1}^{\infty}$ is also a BS-rigidity challenge for $\Gamma$. Since $\Gamma$ is BS-rigid, this challenge has a solution. It follows that $\Gamma$ is stable.
\end{proof}

\section{Proofs}

\subsection{Proof of Propositions \ref{prop:BS-rigidViaHoms} and \ref{prop:BS-rigidViaHomsFlex}}\label{sec:ProofBS-rigidViaHoms}

\cref{prop:BS-rigidViaHoms} is the case $\nu=0$ of \cref{prop:BS-rigidViaHomsFlex}. Hence, it suffices to prove the latter.
As before, $F_S$ denotes the free group on the set $S = \inset{s_1,\dotsc, s_k}$, and we fix a function $\nu:\R_{>0}\times \N\to \N\cup \{0,\infty\}$, monotone non-decreasing in the first argument.

The following lemma is a simple consequence of the triangle inequality. 
\begin{lemma}[{\cite[Lemma 6.9]{BLM} and \cite[Corollary 6.10]{BLM})}]
\label{lem:TV-information-loss}
Let $P'\subset P$ be finite subsets of $F_S$ and take $\ol\sigma,\ol\tau\in\Sym(n)^k$, $n\in \N$. Then
$d_{\TV}(N_{\ol\sigma,P'},N_{\ol\tau,P'})
\leq d_{\TV}(N_{\ol\sigma,P},N_{\ol\tau,P})$
\end{lemma}

\begin{lemma}
\label{lem:disting-via-seqs}
A subset $E$ of $F_S$ is $\nu$-flexibly BS-rigid if and only if the following condition holds: 

$(\star)$ For all sequences $(\ol\sigma^{(\ell)})_{\ell=1}^\infty$ and $(\ol\tau^{(\ell)})_{\ell=1}^\infty$, $\ol\sigma^{(\ell)}\in\Sym(n_\ell)^k$, $\ol\tau^{(\ell)}\in\Sol_E(n_\ell)$, $n_\ell\in\N$, such that $d_{\TV}(N_{\ol\sigma^{(\ell)},P},N_{\ol\tau^{(\ell)},P})\to 0$ for every finite subset $P$ of $F_S$, we have $\ol\sigma^{(\ell)}\in\Sol_E^{<\eps_\ell,\nu\flex}(n_\ell)$ for some $\eps_\ell\to 0$.
\end{lemma}
\begin{proof}
Clearly, if $E$ is $\nu$-flexibly BS-rigid then $(\star)$ holds.
Now, let $(P_m)_{m=1}^{\infty}$ be an ascending sequence of finite subsets of $F_S$ such that $\bigcup_{m=1}^\infty P_m = F_S$.
If $E$ is not $\nu$-flexibly BS-rigid then there is an $\eps>0$ and sequences 
$(\ol\sigma^{(\ell)})_{\ell=1}^\infty$ and $(\ol\tau^{(\ell)})_{\ell=1}^\infty$, $\ol\sigma^{(\ell)}\in\Sym(n_\ell)^k$, $\ol\tau^{(\ell)}\in\Sol_E(n_\ell)$, $n_\ell\to\infty$, such that $d_{\TV}(N_{\ol\sigma^{(\ell)},P_{\ell}},N_{\ol\tau^{(\ell)},P_{\ell}}) < 1/\ell$,
but $\ol\sigma^{(\ell)}\in\Sol_E^{\geq\eps,\nu\flex}(n_\ell)$.
In light of \cref{lem:TV-information-loss}, this implies that $(\star)$ does not hold.
\end{proof}

Let $\Gamma$ be a quotient group of $F_S$, and write $\pi$ for the quotient map.
Recall the following notation from the introduction (here $n\in\N$, and $x$ is sampled uniformly from $[n]$): \begin{inparaenum}\item For $\ol\sigma\in \Sym(n)^k$ and $P\subseteq F_S$, we write $N_{\ol \sigma,P}$ for the distribution, over the power set $2^P$, of $\inset{w\in P\mid w\inparen{\ol \sigma}x=x}$. \item For a function $f:\Gamma\to \Sym(n)$ and $Q\subseteq \Gamma$, we write $N_{f,Q}$ for the distribution, over the power set $2^Q$, of $\inset{\gamma\in Q\mid f(\gamma)x=x}$.\end{inparaenum}

The following proof requires the following related definition: for a function $f\colon\Gamma\to\Sym(n)$ and $P\subset F_S$, write $N_{f,P}$ for the distribution, over the power set $2^P$, of $\inset{w\in P\mid f(\pi(w))x=x}$.

\begin{proof}[Proof of \cref{prop:BS-rigidViaHomsFlex}]
Assume that $E$ is $\nu$-flexibly BS-rigid. Take a sequence of functions $(f_\ell\colon\Gamma\to\Sym(n_\ell))_{\ell=1}^\infty$, $n_\ell\in\N$, such that
\begin{align}
\label{eq:DistingEquivalencePfAsymptoticHom}
    d(f_\ell(\gamma_1\gamma_2),f_\ell(\gamma_1)f_\ell(\gamma_2))\to_{\ell\to\infty}0\qquad\forall\gamma_1,\gamma_2\in\Gamma\ecomma
\end{align}
and such that there is a sequence of homomorphisms $(g_\ell\colon\Gamma\to\Sym(n_\ell))_{\ell=1}^\infty$ satisfying
\begin{align}
\label{eq:DistingEquivalencePfDistChallenge}
    d_{\TV}(N_{f_\ell,Q},N_{g_\ell,Q})\to_{\ell\to\infty}0\qquad\forall Q\subset\Gamma, |Q|<\infty\eperiod
\end{align}

Let $\ol\sigma^{(\ell)}=(f_{\ell}(\pi(s_1),\dotsc,f_{\ell}(\pi(s_k)))\in\Sym(n_{\ell})^k$ and
$\ol\tau^{(\ell)}=(g_{\ell}(\pi(s_1),\dotsc,g_{\ell}(\pi(s_k)))\in\Sol_E(n_\ell)$.
Take a finite subset $P$ of $F_S$, and let $A_{\ell}=\inset{x\in[n_\ell]\mid w(\ol\sigma^{(\ell)})x=f_{\ell}(\pi(w))x~~\forall w\in P}$.
Then \eqref{eq:DistingEquivalencePfAsymptoticHom} implies that $|A_{\ell}|/n_{\ell}\to_{\ell\to\infty}1$. Thus $d_{\TV}(N_{\ol\sigma^{(\ell)},P},N_{f_\ell,P})\to_{\ell\to\infty}0$. On the other hand, \eqref{eq:DistingEquivalencePfDistChallenge} implies that
$d_{\TV}(N_{f_\ell,P},N_{g_\ell,P})\to_{\ell\to\infty}0$.
Thus
$d_{\TV}(N_{\ol\sigma^{(\ell)},P},N_{g_{\ell},P})\to 0$. But $N_{g_{\ell},P}=N_{\ol\tau^{(\ell)},P}$ because $g_{\ell}$ is a homomorphism, and hence $d_{\TV}(N_{\ol\sigma^{(\ell)},P},N_{\ol\tau^{(\ell)},P})\to_{\ell\to\infty} 0$.

Since the above holds for every finite subset $P$ of $F_S$, and since $E$ is $\nu$-flexibly BS-rigid, \cref{lem:disting-via-seqs} implies that $\ol\sigma^{(\ell)}\in\Sol_{E}^{<\eps_{\ell},\nu\flex}(n_\ell)$ for some $\eps_\ell\to 0$. In other words, there is a sequence $(\ol\omega^{(\ell)})_{\ell=1}^\infty$, $\ol\omega^{(\ell)}\in\Sol_E(N_\ell)$, such that $\sum_{i=1}^k d(\sigma^{(\ell)}_i,\omega^{(\ell)}_i) < \eps_\ell$ and $n_\ell \le N_\ell \le n_\ell + \nu(\eps_\ell,n_\ell)$ for all $\ell$. For each $\ell\geq 1$, define a homomorphism $h_\ell\colon\Gamma\to\Sym(N_\ell)$ by letting $h_\ell(\pi(s_i))=\omega^{(\ell)}_i$ for each $1\leq i\leq k$.
Thus $d(f_\ell(\pi(s_i)), h_\ell(\pi(s_i)))=d(\sigma^{(\ell)}_i,\omega^{(\ell)}_i)\to 0$, and so \eqref{eq:DistingEquivalencePfAsymptoticHom} implies that
$d(f_\ell(\gamma),h_\ell(\gamma))\to 0$ for each $\gamma\in\Gamma$. Thus $\Gamma$ is BS-rigid.

Now assume that $\Gamma$ is $\nu$-flexibly BS-rigid.
Take sequences $(\ol\sigma^{(\ell)})_{\ell=1}^\infty$ and $(\ol\tau^{(\ell)})_{\ell=1}^\infty$, $\ol\sigma^{(\ell)}\in\Sym(n_\ell)^k$, $\ol\tau^{(\ell)}\in\Sol_E(n_\ell)$, such that
\begin{align}\label{eq:BS-via-Homs-BS-dist}
    d_{\TV}(N_{\ol\sigma^{(\ell)},P},N_{\ol\tau^{(\ell)},P})\to 0\qquad \forall P\subset F_S, |P|<\infty\eperiod
\end{align}

Fix a function $p\colon\Gamma\to F_S$ such that $\pi(p(\gamma))=\gamma$ for each $\gamma\in\Gamma$, and $p(\pi(s_i))=s_i$ for each $1\leq i\leq k$. Define functions $f_\ell\colon\Gamma\to\Sym(n_\ell)$ by letting $f_\ell(\gamma)=p(\gamma)(\ol\sigma^{(\ell)})$.
Take $\gamma_1,\gamma_2\in\Gamma$ and set $w=p(\gamma_1\gamma_2)p(\gamma_2)^{-1}p(\gamma_1)^{-1}$. Then $w\in\ker\pi$ and thus $w(\ol\tau^{(\ell)})=1_{\Sym(n_\ell)}$ for each $\ell\geq 1$. Hence, \eqref{eq:BS-via-Homs-BS-dist} implies that
$$d(f_{\ell}(\gamma_1\gamma_2),f_{\ell}(\gamma_1)f_\ell(\gamma_2))=
d(w(\ol\sigma^{(\ell)}),1_{\Sym(n_\ell)})\to 0
\eperiod$$
Define homomorphisms $g_{\ell}\colon\Gamma\to\Sym(n_\ell)$ by letting $g_\ell(\pi(s_i))=\tau^{(\ell)}_i$ for each $1\leq i\leq k$.
Then $g_\ell(\gamma)=p(\gamma)(\ol\tau^{(\ell)})$ for each $\gamma\in\Gamma$.
Take a finite subset $Q$ of $\Gamma$ and let $P=p(Q)$.
Then
$$d_{\TV}(N_{f_\ell,Q},N_{g_\ell,Q})=
d_{\TV}(N_{\ol\sigma^{(\ell)},P},N_{\ol\tau^{(\ell)},P})
\to_{\ell\to\infty}0\eperiod
$$
Since $\Gamma$ is $\nu$-flexibly BS-rigid, we deduce that there is a sequence $\eps_\ell\to 0$ and a sequence  $(h_\ell\colon\Gamma\to\Sym(N_\ell))_{\ell=1}^\infty$ of homomorphisms, with $n_\ell \le N_\ell\le n_\ell + \nu(\eps_\ell,n_\ell)$, 
such that
$d(f_\ell(\gamma), h_\ell(\gamma))\to 0$ for each $\gamma\in\Gamma$.
Define $\ol\omega^{(\ell)}\in\Sol_{E}(N_\ell)$ by letting $\ol\omega^{(\ell)}_i=h_\ell(\pi(s_i))$ for each $1\leq i\leq k$. Then, recalling that $p(\pi(s_i))=s_i$ for all $1\leq i\leq k$, we see that
$$
d(\ol \sigma^{(\ell)}, \ol \omega^{(\ell)}) = \sum_{i=1}^k d(\sigma^{(\ell)}_i,\omega^{(\ell)}_i)=
\sum_{i=1}^k d(f_{\ell}(\pi(s_i)), h_{\ell}(\pi(s_i)))< \eps'_\ell
$$
for some $\eps'_\ell\to 0$. Let $\eps''_\ell = \max\inset{\eps_\ell, \eps'_\ell}$. Then $\eps''_\ell \to 0$, and we have
$d(\ol\sigma^{(\ell)}, \ol \omega^{(\ell)})< \eps''_\ell$. Also, $n_\ell\le N_\ell \le n_\ell + \nu(\eps''_\ell, n_\ell)$ since $\nu$ is nondecreasing in the first argument. Hence $\ol\sigma^{(\ell)} \in \Sol^{<\eps''_\ell, \nu\flex}(n_\ell)$, 
and thus $E$ is BS-rigid by \cref{lem:disting-via-seqs}.
\end{proof}

\subsection{Proof of \cref{prop:IRSFormulation}}\label{sec:ProofIRSFormulation}
We require the following lemma.
\begin{lemma}
[{\cite[Lemma\ 7.6]{BLT}}]
\label{lem:amplification-sparsification}
Let $(X_n)_{n=1}^\infty$ be a sequence of finite $\Gamma$-sets, $|X_n|\to\infty$, $X_n\to\mu$, $\mu\in\IRS(\Gamma)$.
Let $(m_n)_{n=1}^\infty$ be a sequence of positive integers, $m_n\to\infty$. Then there is a sequence $(Y_n)_{n=1}^\infty$ of finite $\Gamma$-sets such that $|Y_n|=m_n$ and $Y_n\to\mu$.
\end{lemma}

\begin{proof}[Proof of \cref{prop:IRSFormulation}]

1. See \cite[Lemma 7.4]{BLT} for a proof in the case $\nu=0$. The proof of the general case is almost identical.

2. We prove the claim for $\nu$-flexible BS-rigidity (the case of BS-rigidity is the special case $\nu=0$).

Recall that we realize $\Gamma$ as a quotient of a free group $F_S$, $S=\inset{s_1,\dotsc,s_k}$, and write $\Gamma=\langle S\mid E\rangle$ for some $E\subset F_S$.

Assume that $\Gamma$ is not $\nu$-flexibly BS-rigid.
Then $E$ is not $\nu$-flexibly BS-rigid by \cref{prop:BS-rigidViaHoms}.
By \cref{lem:disting-via-seqs}, there is $\eps>0$
and sequences $(\overline{\sigma}^{(\ell)})_{\ell=1}^\infty$ and $(\overline{\tau}^{(\ell)})_{\ell=1}^\infty$,
$\overline{\sigma}^{(\ell)}\in\Sol_{E}^{\geq \eps, \nu\flex}(n_\ell)$,
$\overline{\tau}^{(\ell)}\in\Sol_{E}(n_\ell)$, $n_\ell\in\N$,
such that 
$d_{\TV}(N_{\overline{\sigma}^{(\ell)},P}, N_{\overline{\tau}^{(\ell)},P})\to 0$ for every finite subset $P$ of $F_S$.
Now, in general, a tuple $\overline{\sigma}\in\Sym(n)^k$, $n\in\N$, gives rise to an action of $F_S$ on $[n]$, given by sending each basis element $s_i$ of $F_S$ to the permutation $\sigma_i$. If $\overline{\sigma}\in\Sol_{E}(n)$ then this $F_S$-action factors through $\Gamma$.

Using this, and the compactness of $\IRS(F_S)$, we see that there is a sequence $(X_\ell)_{\ell=1}^\infty$ of finite $F_S$-sets and a sequence $(Y_\ell)_{\ell=1}^\infty$ of finite $\Gamma$-sets, $|X_\ell|=|Y_\ell|$, such that $(X_\ell)_{\ell=1}^\infty$ and $(Y_\ell)_{\ell=1}^\infty$ converge to the same measure $\mu\in\IRS(\Gamma)$ (which is co-sofic in $\Gamma$ because each $Y_\ell$ is a finite $\Gamma$-set), but $(X_\ell)_{\ell=1}^\infty$ does not have a $\nu$-flexible solution. In other words, $(X_\ell)_{\ell=1}^\infty$ is a BS-rigidity challenge for $\Gamma$ that has no $\nu$-flexible solution.

Now, assume that $\Gamma$ has a BS-rigidity challenge $(X_\ell)_{\ell=1}^\infty$ that has no $\nu$-flexible solution.
We may assume that there is $\eps>0$ such that  for all $\ell\geq 1$ we have $d_S(X_\ell,Z_\ell)\geq \eps$ for every $\Gamma$-set $Z_\ell$ such that
$|X_\ell|\leq|Z_\ell|\leq |X_\ell|+\nu(\eps,|X_\ell|)$.

Let $\mu\in\IRS(\Gamma)$ be the $\Gamma$-co-sofic measure such that $X_\ell\to\mu$.
By \cref{lem:amplification-sparsification}, there is a sequence $(Y_\ell)_{\ell=1}^\infty$ of $\Gamma$-sets, $|Y_\ell|=|X_\ell|$, such that $Y_\ell\to\mu$. Define $n_\ell\coloneqq |X_\ell|=|Y_\ell|$
In general, a finite $F_S$-set $X$ gives rise to a tuple $\overline{\sigma}=(\sigma_1,\dotsc,\sigma_k)$ of permutations, $\sigma_i\in\Sym(X)$, defined by letting $\sigma_i x=s_i x$ for all $1\leq i\leq k$ and $x\in X$.

Using this, $X_\ell$ and $Y_\ell$ give rise to tuples $\overline{\sigma}^{(\ell)}\in\Sol_{E}^{\geq \eps,\nu\flex}(|X_\ell|)$ and $\overline{\tau}^{(\ell)}\in\Sol_{E}(|Y_\ell|)$
such that $d_{\TV}(N_{\overline{\sigma}^{(\ell)},P}, N_{\overline{\tau}^{(\ell)},P})\to 0$ for every finite subset $P$ of $F_S$, and so $E$ is not $\nu$-flexibly BS-rigid. Hence $\Gamma$ is not $\nu$-flexibly BS-rigid by \cref{prop:BS-rigidViaHoms}.

The proof of (3) is very similar to the proof of (2).
\end{proof}

\subsection{Proofs of claims from \cref{sec:GeneralProps}}\label{sec:ProofsGeneralProps}

\begin{proof}[Proof of \cref{prop:DistinguishabilityByFiniteResidual}]
In general, for a finitely-generated group $\Delta=\inang{S\mid E}$, $E\subset F_S$, the question of whether $\Delta$ is ($\nu$-flexibly) BS-rigid depends only on the sequence $(\Sol_E(n))_{n=1}^\infty$.

Let $\Gamma=\inang{S\mid E}$ and $\Gamma/R(\Gamma)=\inang{S\mid E'}$ be presentations of $\Gamma$ and $\Gamma/R(\Gamma)$, $E\subset E'\subset F_S$. The claim follows since $\Sol_E(n)=\Sol_{E'}(n)$ for all $n\geq 1$.
~
\end{proof}

\begin{lemma}
\label{lem:IRS-fg-normal-sgp}
Let $\Gamma$ be a finitely-generated group,
and let $N$ be a finitely-generated subgroup of $\Gamma$ such that $N\lhd \Gamma$.
Let $(X_{n})_{n=1}^\infty$
be a sequence of finite $\Gamma$-sets that converges to a measure $\mu$ which belongs to the space $\IRS(\Gamma/N)$, viewed as a subspace of $\IRS(\Gamma)$.
Then there is a sequence $(Z_{n})_{n=1}^\infty$ of $\Gamma/N$-actions
such that $(X_{n})_{n=1}^\infty$ and $(Z_{n})_{n=1}^\infty$ are equivalent
and $|Z_n|=|X_n|$ for all $n$.
\end{lemma}
\begin{proof}
For $n\in\N$, let $X'_n$ be the subset of $X_n$ consisting of all points $x\in X_n$ such that $N\leq\stab_{\Gamma}(x)$.
Then $X'_n$ is a union of orbits on $X_n$ because $N$ is normal in $\Gamma$.
Let $Z_n = X_n \times \{0\}$ be a copy of $X_n$, equipped with the following $\Gamma$-action: For all $(\gamma,x)\in\Gamma\times X_n$, $\gamma (x,0)=(\gamma x, 0)$ if $x\in X'_n$ and $\gamma(x,0)=(x,0)$ otherwise. The action of $\Gamma$ on $Z_n$ clearly factors through $\Gamma/N$. To prove that $(X_n)_{n=1}^\infty$ and $(Z_n)_{n=1}^\infty$ are equivalent, it remains to show that $|X'_n|/|X_n|\to 1$.

Let $\mu_n\in\IRS(\Gamma)$ be the random stabilizer of $X_n$, and take a finite generating set $A\subset\Gamma$ for $N$. Then, recalling the notation from \cref{sec:IRS}, we have
$\mu_n(C_{A,A})=|X'_n|/|X_n|$ and
$\mu_n(C_{A,A})\to\mu(C_{A,A})$. But $\mu(C_{A,A})=1$ because $\mu$ belongs to $\IRS(\Gamma/N)$.
\end{proof}

\begin{proof}[Proof of \cref{prop:quotient-fg}]
Take a stability (resp. BS-rigidity) challenge $(X_n)_{n=1}^\infty$ for $\Gamma/N$, and write $\mu\in\IRS(\Gamma/N)$ for its limit measure.
Then $(X_{n})_{n=1}^\infty$
is also a stability (resp. BS-rigidity) challenge for $\Gamma$, and as such, it has a ($\nu$-flexible) solution
$(Y_{n})_{n=1}^\infty$, where each $Y_n$ is a $\Gamma$-set.

By Lemma \ref{lem:IRS-fg-normal-sgp}, there is a sequence $(Z_{n})_{n=1}^\infty$ of finite $\Gamma$-sets, such that $(Y_{n})_{n=1}^\infty$ and
$(Z_{n})_{n=1}^\infty$ are equivalent and $|Z_n|=|Y_n|$.
Thus $(Z_{n})_{n=1}^\infty$ is a ($\nu$-flexible) solution for
$(X_{n})_{n=1}^\infty$.
\end{proof}

\begin{proof}[Proof of \cref{prop:FiniteIndexSubgroup}]
Our claim about flexible stability is an immediate generalization of \cite[Lemma 3.3]{Ioana}. For flexible BS-rigidity, our proof slightly differs, but follows from the same idea of inducing actions (and approximate actions) from $H$ to $\Gamma$. We sketch the details below.

Fix a function $s\colon \Gamma/H\to\Gamma$ such that $s(1_{\Gamma}H)=1_\Gamma$ and $s(\gamma H)\in \gamma H$ for every $\gamma\in \Gamma$.
Then $c\colon\Gamma\times\Gamma/H
\to H$, given by $c(\gamma_1,\gamma_2 H)=s(\gamma_1 \gamma_2 H)^{-1} \gamma_1 s(\gamma_2 H)$, is a co-cycle for the left-multiplication action of $\Gamma$ on $\Gamma/H$, i.e.,
$c(\gamma_1 \gamma_2, \gamma_3 H)=c(\gamma_1, \gamma_2 \gamma_3 H)c(\gamma_2, \gamma_3 H)$ for all $\gamma_1,\gamma_2,\gamma_3\in\Gamma$.

Take a sequence of functions $(f^H_\ell\colon H\to\Sym(n_\ell))_{\ell=1}^\infty$, $n_\ell\to\infty$, such that $d(f^H_\ell(\gamma_1\gamma_2),f^H_\ell(\gamma_1)f^H_\ell(\gamma_2))\to 0$ for all $\gamma_1,\gamma_2\in H$, and assume that there is a sequence of homomorphisms $g^H_{\ell}\colon H\to\Sym(n_\ell)$ such that
\begin{equation}\label{eq:FiniteIndexSubgropfAndgSimilar}
d_{\TV}(N_{f^{H}_\ell,Q},N_{g^{H}_\ell,Q})\to 0~~~\text{for every finite }Q\subseteq H\eperiod
\end{equation}

Define
$f^\Gamma_\ell\colon\Gamma\to\Sym(\Gamma/H\times [n_\ell])$ and
$g^\Gamma_\ell\colon\Gamma\to\Sym(\Gamma/H\times [n_\ell])$
by letting
$f^\Gamma_\ell(\gamma_1)(\gamma_2 H,x)=(\gamma_1 \gamma_2 H, f^{H}_{\ell}(c(\gamma_1, \gamma_2 H))x)$ and
$g^\Gamma_\ell(\gamma_1)(\gamma_2 H,x)=(\gamma_1 \gamma_2 H, g^{H}_{\ell}(c(\gamma_1, \gamma_2 H))x)$.
Note that for every $h\in H$, the restriction of the permutation $f_{\ell}^{\Gamma}(h)\in\Sym(\Gamma/H\times [n_\ell])$ to $\inset{1_\Gamma H}\times [n_\ell]$ is identified in the natural way with the permutation $f_{\ell}^{H}(h)\in\Sym(n_\ell)$.

By direct computation, one can see that $d\inparen{f_\ell^\Gamma(\gamma_1\gamma_2), f_\ell^\Gamma(\gamma_1)f_\ell^\Gamma(\gamma_2)}\to 0$ for all $\gamma_1,\gamma_2\in \Gamma$ (see \cite{Ioana}).

For every finite subset $Q$ of $\Gamma$, let $Q_H$ denote the subset $c(Q\times\Gamma/H)$ of $H$.
Now, for $\gamma'\in\Gamma$ and $x\in [n_\ell]$, we have
$\stab_{f^{\Gamma}_{\ell}}(\gamma' H,x)=
\inset{\gamma\in \gamma' H\gamma'^{-1} \mid c(\gamma,\gamma' H)\in\stab_{f^H_\ell}(x)}$, and thus
$\stab_{f^{\Gamma}_{\ell}}(\gamma' H,x)\cap Q$ is determined by
$\stab_{f^H_\ell}(x)\cap Q_H$ and $\gamma' H$ for every finite subset $Q$ of $\Gamma$. Similarly, $\stab_{g^{\Gamma}_{\ell}}(\gamma' H,x)\cap Q$ is determined in the same manner by $\stab_{g^H_\ell}(x)\cap Q_H$ and $\gamma' H$. Thus, \eqref{eq:FiniteIndexSubgropfAndgSimilar} implies that $d_{\TV}(N_{f^{\Gamma}_\ell,Q},N_{g^{\Gamma}_\ell,Q})\to 0$ for every finite $Q\subseteq \Gamma$.

Since $\Gamma$ is $\nu$-flexibly BS-rigid, and since $(g^\Gamma_\ell)_{\ell=1}^\infty$ are homomorphisms, there is a sequence $(\eps_\ell)_{\ell=1}^\infty$ with $\eps_\ell\to 0$, and homomorphisms $\inparen{h_\ell^{\Gamma}\colon\Gamma\to \Sym(Y_\ell)}_{\ell=1}^\infty$ where $\Gamma/H\times [n_\ell]\subseteq Y_\ell$ and $Y_\ell$ is a finite set with $|Y_\ell|\le m_\ell + \nu(\eps_\ell,m_\ell)$ ($m_\ell = \inabs{\Gamma/H\times [n_\ell]}$), such that $d\inparen{f_\ell^{\Gamma}(\gamma),h_\ell^{\Gamma}(\gamma)}\to 0$ for all $\gamma \in \Gamma$. In particular, $d\inparen{f_\ell(\gamma),h_\ell(\gamma)}\to 0$ for every $\gamma\in H$, proving the claim
(noting that
$|Y_\ell|-n_\ell=([\Gamma:H]-1)n_\ell+\nu(\eps_\ell,[\Gamma:H]n_\ell)$).
\end{proof}

\begin{proof}[Proof of \cref{prop:RF-not-weakly-stable}]
Let $\Gamma$ be finitely-generated residually-finite ($\nu$-flexibly) BS-rigid group.
Let $(X_n)_{n=1}^\infty$ be a sofic approximation for $\Gamma$. That is, each $X_n$ is a finite $F_S$-set, and $(X_n)_{n=1}^\infty$ converges to the Dirac measure $\mu$ concentrated on the trivial subgroup of $\Gamma$.
But $\mu$ is co-sofic in $\Gamma$ because $\Gamma$ is residually finite.
Thus $(X_n)_{n=1}^\infty$ is a BS-rigidity challenge for $\Gamma$, and hence it has a ($\nu$-flexible) solution because $\Gamma$ is ($\nu$-flexibly) BS-rigid.
\end{proof}

\subsection{Proof of \cref{thm:NegativeKazhdan}}\label{sec:ProofNegativeKazhdan}
Recall that for a group $\Gamma$, generated by a finite set $S$, property $\ptau$ can be defined simply as follows: $\Gamma$ has property $\ptau$ if there is $\kappa>0$ such that for every finite-dimensional unitary representation $\rho\colon\Gamma\to\calU(\calH)$, if $\rho$ factors through a finite quotient of $\Gamma$ and $v\in\calH$ satisfies $\|\rho(s)v-v\|\leq \eps\|v\|$ for all $s\in S$, $\eps>0$, then $\|v-Pv\|\leq\frac{\eps}{\kappa}\|v\|$, where $P\colon\calH\to\calH^\Gamma$ is the orthogonal projection onto the subspace $\calH^\Gamma$ of $\Gamma$-invariant vectors.

Part 2 of \cref{thm:NegativeKazhdan} is proved in \cite{BeckerLubotzky}. We shall prove Part 1 here.
\begin{proof}[Proof of \cref{thm:NegativeKazhdan}(1)]
Since $\Gamma$ has infinitely many finite quotients, there is a convergent sequence $(X_n)_{n=1}^\infty$ of transitive finite $\Gamma$-sets, $|X_n|\to\infty$.
Write $\mu\in\IRS(\Gamma)$ for the limit measure of $X_n$. Clearly, $\mu$ is co-sofic.
For $n\in\N$, let $Y_n=(X_n\setminus\{x_0^n\})\times\{0\}$, where $x_0^n$ is an arbitrary element of $X_n$. Define an $F_S$-action on $Y_n$ as follows: Take $(s,x)\in S\times(X_n\setminus\{x_0^n\})$. If $sx\neq x_0^n$, let $s(x,0)=(sx,0)$. Otherwise, let $s(x,0)=(s^{2}x,0)$.
Then $Y_n\to\mu$, and thus $(Y_n)_{n=1}^\infty$ is a BS-rigidity challenge for $\Gamma$.
Assume for the sake of contradiction that $(Y_n)_{n=1}^\infty$ has a solution $(Z_n)_{n=1}^\infty$. That is, each $Z_n$ is a finite $\Gamma$-set, $|Z_n|=|Y_n|$, and $d_S(Y_n,Z_n)\to 0$.
Consider the $\Gamma$-modules $\C[X_n]$ and $\C[Z_n]$ of $\C$-linear combinations of elements of $X_n$ and $Z_n$, respectively, where $\Gamma$ acts by multiplication from the left. Also consider the $\Gamma$-action on the space $L\coloneqq\Hom_{\C}(\C[Z_n],\C[X_n])$, consisting of all linear maps from $\C[Z_n]$ to $\C[X_n]$, given by $\gamma f=\gamma\circ f\circ\gamma^{-1}$. Endowing $L$ with the unique Hermitian product such that the set $\inset{E_{x,z}\mid x\in X_n, z\in Z_n}$ is an orthonormal basis, the $\Gamma$-action on $L$ becomes a unitary representation (here $E_{x,z}$ is the linear map from $\C[Z_n]$ to $\C[X_n]$ that sends $z$ to $x$ and sends every $z'\in Z_n\setminus\inset{z}$ to $0$).
Note that the subspace of $\Gamma$-invariant vectors in $L$ is the space of morphisms of $\Gamma$-representations from $\C[Z_n]$ to $\C[X_n]$.
Let $f_n\colon \C[Z_n]\to \C[X_n]$ be the linear map such that $f_n((x,0))=x$ for every $x\in X_n\setminus\{x_0^{n}\}$.
Now, \cite[Proposition 2.4]{BeckerLubotzky} deals with this exact construction (without assuming that $\Gamma$ has property $\ptau$), and says that $\|f_n-h_n\|\geq\frac{1}{\sqrt{2}}\|f_n\|$ for every morphism of representations $h_n\colon \C[Z_n]\to\C[X_n]$.
But $\|sf_ns^{-1}-f_n\|\to 0$ for every $s\in S$, in contradiction to the assumption that $\Gamma$ has property $\ptau$.
\end{proof}

\section*{Acknowledgements}
O.B.\ has received funding from the European Research Council (ERC) under the European Union's Horizon 2020 research and innovation programme (grant agreement No. 803711).
A.L.\ is supported by a grant from the Institute for Advanced Study and by the European Research Council (ERC) under the European Union Horizon 2020 research and innovation program (Grant No. 692854).
J.M.\ is partially supported by NSF grant CCF-1814603.

\bibliography{testabilityInGroupTheory.bib}

\end{document}